\documentclass[a4paper,oneside,reqno]{amsart}

\usepackage[utf8]{inputenc}
\usepackage[english]{babel}
\usepackage{amssymb}
\usepackage{amsmath}
\usepackage{amsthm}
\usepackage{amsfonts}
\usepackage{amsaddr}
\usepackage{enumerate}

\usepackage{graphicx}
\usepackage{subfig}
\usepackage{float}
\usepackage{listings}
\usepackage{hyperref}
\usepackage[margin = 4 cm]{geometry}
\usepackage{dsfont}
\usepackage{color}

\usepackage{tikz,pgf}
\usepackage{geometry}
\usetikzlibrary{arrows}
\usepackage{pgfplots}
\pgfplotsset{compat=1.5}
\usepackage{mathrsfs}
\definecolor{BLUE}{rgb}{0.30196078431372547,0.30196078431372547,1}
\definecolor{RED}{rgb}{1,0,0}

\usepackage{stmaryrd}
\usepackage{bbold}
\usepackage{mathtools,cases}
\usepackage{physics}
\usepackage{amsmath}
\usepackage{version}
\usepackage{empheq}

\newcommand{\scalar}[2]{\langle#1,#2\rangle}

\newcommand{\intint}[2]{[\![#1,#2]\!]}
\newcommand{\len}{\text{len}}

\theoremstyle{plain}
\newtheorem{thm}{Theorem}[section]
\newtheorem{prop}[thm]{Proposition}
\newtheorem{lem}[thm]{Lemme}
\newtheorem{cor}[thm]{Corollary}
\newtheorem{rem}[thm]{Remark}
\newtheorem{ass}{Assumption}
\newtheorem{defn}{Definition}[section]

\theoremstyle{definition}

\newtheorem{exmp}{Example}[section]

\theoremstyle{remark}

\newtheorem*{prf}{\textit{Proof}}
\usepackage[justification=centering]{caption}
\DeclareCaptionFormat{sanslabel}{#3}%

\date{}

\begin{document}

\title{Flocking of the Cucker-Smale and Motsch-Tadmor models on general weighted digraphs via a probabilistic method}
\xdef\shorttitle{Flocking of CS and MT models via a probabilistic method}
\author{Adrien Cotil}
\address{UMR MISTEA, \\ Univ Montpellier, INRAE, Institut Agro, \\ 
34060 Montpellier, France.}

\maketitle

\begin{abstract}
In this paper, we discuss the flocking phenomenon for the Cucker-Smale and Motsch-Tadmor models in continuous time on a general oriented and weighted graph with a general communication function. We present a new approach for studying this problem based on a probabilistic interpretation of the solutions. We provide flocking results under four assumptions on the interaction matrix and we highlight how they relate to the convergence in total variation of a certain Markov jump process. Indeed, we refine previous results on the minimal case where the graph admits a unique closed communication class. Considering the two particular cases where the adjacency matrix is scrambling or where it admits a positive reversible measure, we improve the flocking condition obtained for the minimal case. In the last case, we characterise the asymptotic speed. We also study the hierarchical leadership case where we give a new general flocking condition which allows to deal with the case $\psi(r)\propto(1+r^2)^{-\beta/2}$ and $\beta\geq1$. For the Motsch-Tadmor model under the hierarchical leadership assumption, we exhibit a case where the flocking phenomenon occurs regardless of the initial conditions and the communication function, in particular even if $\beta\geq1$.
\end{abstract}

\tableofcontents

\section{Introduction}\label{sec:intro}

Recent studies on the theoretical understanding of collective behaviours of certain living systems have recently received much attention. See for example \cite{couzin2003self} for a study of self-organisation phenomena in vertebrates or in a more general case with \cite{parrish1999complexity}. See also \cite{verdiere2014mathematical} for a model of human behaviours during disaster scenarios. These types of models can also be used to simulate the group movements, like in \cite{reynolds1987flocks} in which various bird flight simulation techniques are explored. These models usually include three effects: the long-distance attraction, the short-distance repulsion and the alignment. These effects are known as the \textit{First Principles of Swarming}. In \cite{park2010cucker}, such a model is studied with the addition of a collision avoidance term. This system of interactions has many variants, particularly for adapting it to particular contexts. For instance, it is used in \cite{couzin2003self}, \cite{barbaro2009discrete} and \cite{bernoff2011primer} to model the movement of a bird, a fish and an insect population respectively. \\

One of the first attempts to model the alignment was proposed in \cite{vicsek1995novel} by Viscek et al. in 1995. Theoretical studies are given in \cite{cao2008reaching,cao2008reachinggraph}. A more recent model has been introduced in \cite{cucker2007emergent} by Cucker and Smale in 2007 and many variants have since been proposed. Notably, in 2011, Motsch and Tadmor suggest in \cite{motsch2011new} to weight the influence of an agent $j$ on a agent $i$ by the total influence exerted on $i$. In the present article, we study Cucker-Smale and Motsch-Tadmor type models in which we add the assumption that each individual does not necessarily interact uniformly with all the others. This assumption allows us to take into account certain social interactions that exist within a group of individuals. \\

\textbf{Model and Notations:} Let $N\in\mathbb{N}^*=\{1,2,3,\dots\}$ be the number of individuals and $x_i(t)\in\mathbb{R}^d$ and $v_i(t)\in\mathbb{R}^d$ be the position and the velocity of agent $i\in\intint{1}{N}=\{1,2,\dots,N\}$ at time $t\in\mathbb{R}_+=[0,+\infty)$. We will study the solutions of the following system of ODEs: $\forall (i,t)\in\intint{1}{N}\times\mathbb{R}_+,$

\begin{equation}\label{eq:model}
    \left\{\begin{aligned}
        &\frac{dx_i}{dt}(t) = v_i(t), \\
        &\frac{dv_i}{dt}(t) = \alpha\sum_{j=1}^NQ_t(i,j)(v_j(t)-v_i(t)),
    \end{aligned}\right.
\end{equation}

where $\alpha>0$ and for all $t\geq0$, $Q_t$ is an $N\times N$-transition rate matrix, that is a matrix which verifies
\begin{equation}\label{eq:TRM}
    \left\{\begin{aligned}
        &\forall i\ne j,\ Q_t(i,j)\geq0, \\
        &\forall i\in\intint{1}{N},\ \sum_{j=1}^NQ_t(i,j)=0.
    \end{aligned}\right.
\end{equation}
Note that \eqref{eq:model} include a wide range of alignment models and notably the case where $Q_t(i,j)$ depends on $(x_i(t),v_i(t))_{i\in\intint{1}{N}}$. We will address this model in the following two cases: for all $t\geq0$ and $i\ne j$, 

\begin{subequations}
\begin{alignat}{2}
    &Q_t(i,j)= A_{ij}\psi(\norm{x_j(t)-x_i(t)}_2),\label{eq:CS} \\
    &Q_t(i,j)=\frac{A_{ij}\psi\left(\norm{x_i(t)-x_j(t)}_2\right)}{a_i+\sum_{k\ne i}A_{ik}\psi\left(\norm{x_i(t)-x_k(t)}_2\right)},\label{eq:MT}
\end{alignat}
\label{eq:assQ}
\end{subequations}

where $A\in\mathbb{R}_+^{N\times N}$ is the interaction matrix, $\psi:\mathbb{R}_+\to\mathbb{R}_+$ is the communication function and $a\in\mathbb{R}^N_+$ satisfying $a_i>0$ if for all $j\ne i,\ A_{ij}=0$. We also assume that $\norm{\psi}_\infty=\sup\{\psi(r)\ |\ r\in\mathbb{R}_+\}\leq1$. Furthermore, we can eventually refer to the interaction graph to designate the graph induced by $A$ (i.e. the graph $\mathcal{G}=(\mathcal{V},\mathcal{E})$ with $\mathcal{V}=\intint{1}{N}$ and $\mathcal{E}=\{(i,j)\in\intint{1}{N}^2\ |\ A_{ij}>0\}$). \\

If $Q_t(i,j)$ satisfies \eqref{eq:CS} and $A_{ij}=C$ for a certain constant $C>0$, we obtain the original Cucker-Smale model introduced in \cite{cucker2007emergent}. Similarly, if $Q_t(i,j)$ satisfies \eqref{eq:MT} and if $A_{ij}=a_i=C$ we obtain the original Motsch-Tadmor model introduced in \cite{motsch2011new}. See \cite{carrillo2017review} for a review of these models and some of their variants. \\

One of the most important question related to Model \eqref{eq:model} is its long time behaviour. In certain cases, one can observe that individuals tend to move at the same speed in the same direction. If a solution of Equation \eqref{eq:model} tends to such a profile, we talk about flocking. More precisely, we say that a solution of \eqref{eq:model} flocks if for all $i\in\intint{1}{N}, \lim_{t\to+\infty}v_i(t)=v^*$ where $v^*\in\mathbb{R}^d$ and for all $(i,j)\in\intint{1}{N}^2,\ \sup_{t\geq 0}\norm{x_i(t)-x_j(t)}_2<+\infty$ where $\norm{.}_2$ is the standard euclidean norm on $\mathbb{R}^d$. See also \cite{cattiaux2018stochastic} for different definitions of the flocking phenomenon for stochastic variants of Model \eqref{eq:model} and for a review of flocking results on such models. \\

The flocking phenomenon has received a lot of interest from various research communities and has been the subject of a large number of publications. It is known that the explicit form of the communication function is not the essence and that only the detailed form of its lower bound is important. Thus, all the results of the present article can easily be generalised to $Q_t(i,j)$ higher than the right-hand side of \eqref{eq:CS} or \eqref{eq:MT}. On the other hand, the key feature of this function is its rate of decay. For instance, to our knowledge, the sharpest flocking condition for the original Cucker-Smale model (i.e $A_{ij}=C$), given in \cite{carrillo2017review}, is 
\begin{equation}\label{eq:flockoptiCS}
    V(0)<\int^{+\infty}_{X(0)}\psi(r)\ dr,
\end{equation}
where $X(0)=\sup_{i,j}\norm{x_i(0)-x_j(0)}_2$ and $V(0)=\sup_{i,j}\norm{v_i(0)-v_j(0)}_2$. In particular the flocking phenomenon happens for all initial conditions if
\begin{equation*}
    \int^{+\infty}\psi(r)\ dr=+\infty.
\end{equation*}
A similar result is obtained for the original Motsch-Tadmor model (i.e $A_{ij}=a_i=C$) in \cite{motsch2014heterophilious,choi2016cucker}. The Cucker-Smale model with a general interaction matrix has also been address in the literature. It is known that the minimal assumption on the matrix $A$ to hope to observe the flocking phenomenon is the existence of a unique closed class in the interaction graph. Equivalently, for all pair of individuals $(i,j)$ such that $i\ne j$, $i$ leads to $j$ (i.e. there exists a path from $i$ to $j$) or $j$ leads to $i$ or there exists another individual $k$ such that $i$ leads to $k$ and $j$ leads to $k$. See \cite{dong2016flocking} for a study of the flocking phenomenon under this assumption or \cite{li2010cucker} under the additional assumption that the unique closed class contains only one individual. Finally, it is possible to make stronger assumptions about the interaction matrix to refine the flocking conditions. For instance, in \cite{cucker2007mathematics}, the authors study the case where $A$ is irreducible and symmetric. In \cite{shen2008cucker}, a non symmetric interaction profile is considered where an individual $i$ is influenced by an individual $j$ only if $j$ is superior to $i$ in the hierarchy. This assumption is named hierarchical leadership assumption. \\

Other even more general models have also been studied in the literature. For instance, in \cite{shu2021anticipation}, the authors assume that $Q_t(i,j)$ is a symmetric positive-definite matrix whose all the eigenvalues $\lambda_{ij}^m(t)$ satisfy: for all $(i,j)\in\intint{1}{N}^2$ such that $i\ne j$,

$$
\forall m\leq d,\ \psi(\norm{x_i(t)-x_j(t)}_2) \leq \lambda_{ij}^m(t) \leq K\,
$$

where $\psi(r)$ given below by \eqref{eq:comini}, $K>1$ and $\beta\in[0,1)$. Using a hypocoercivity argument, they prove in particular that the unconditional flocking occurs for $\beta<2/3$ for this model. Above all, they study this phenomenon in a much wider context by adding an attractive-repulsive potential. As it is explained in the introduction of \cite{shu2021anticipation}, this model can be seen as an aggregation dynamics with anticipation by replacing $Q_t(i,j)$ by $\overline{DU}_{ij}$ which are the ‘intermediate’ Hessians of the previous potential applied in $\norm{x_i(t)-x_j(t)}_2$. This work therefore highlights that an anticipation dynamic with only attractive-repulsive interactions can lead to the flocking phenomenon. \\

In the present article, we provide a unifying framework to address the flocking phenomenon by seeing Equation \eqref{eq:model} as a Kolmogorov equation. In other words, the solution of Equation \eqref{eq:model} can be interpreted as the expectation of a certain function of a certain Markov process. More precisely, we have the following result, which we will present in more detail in section \ref{subsec:probint} (see Theorem \ref{thm:probint}).  

\begin{thm}\label{thm:ProbInt_intro}
Let $(x_i(t),v_i(t))_{i\in\intint{1}{N},t\in\mathbb{R}_+}$ be a solution of Equation \eqref{eq:model}. For all $m\in\intint{1}{d}$, let $f_m(i)=v_i^m(0)$ be the $m$-th coordinate of the initial velocity of agent $i$. Then, there exists a family of time-inhomogeneous Markov jump processes $(Y^{(T)})_{T>0}$ such that for all $T>0$, $Y^{(T)}_t$  is define on $[0,T]$ and for all $t\in[0,T]$, 
\begin{equation*}
    v^m_i(t)=\mathbb{E}\left(f_m\left(Y^{(T)}_T\right)\ |\ Y^{(T)}_{T-t}=i\right).
\end{equation*}
\end{thm}

A direct consequence of Theorem \ref{thm:ProbInt_intro} is that the flocking phenomenon is linked to the convergence in total variation of a family of Markov jump processes. In this article, we provide flocking results, using this probabilistic interpretation under four assumptions on the interaction matrix $A$ :

\begin{itemize}
    \item The case where $A$ admits a unique closed class, studied in \cite{dong2016flocking,dong2019emergent,dong2020stochastic}, which we refer to as the general leadership case.
    \item The case where $A$ is irreducible and reversible, which is a generalisation of the symmetric case studied in \cite{cucker2007mathematics}.
    \item The case where $A$ is scrambling, which means that the interaction graph is strongly connected. To our knowledge, this assumption was not addressed at this level of generality but it can be seen as a generalisation of the case $A_{ij}>0$ for all $i\ne j$.
    \item The hierarchical leadership case, studied notably in \cite{shen2008cucker,cucker2009critical,dalmao2011cucker}.
\end{itemize}

As an illustration, we provide a simplified version of the flocking conditions we obtain in the last case.

\begin{thm}\label{thm:HL_intro}
Let $(x_i(t),v_i(t))_{i\in\intint{1}{N},t\in\mathbb{R}_+}$ be a solution of Equation \eqref{eq:model} with $Q_t$ given by \eqref{eq:CS}. Let us define $X(0)=\sup_{i,j}\norm{x_i(0)-x_j(0)}$ and $V(0)=\sup_{i,j}\norm{v_i(0)-v_j(0)}$. Let us also assume that $A$ satisfies the hierarchical leadership assumption. Then the flocking phenomenon occurs if  
\begin{equation*}
\begin{aligned}
    &V(0)<C_{HL}\sup_{r\geq X(0)}(r-X(0))\psi(r), &\text{for Model \eqref{eq:CS}.} \\
    &V(0)<M_{HL}\sup_{r\geq X(0)}(r-X(0))\frac{(\Bar{a}+A_*)\psi(r)}{\Bar{a}+A_*\psi(r)}, &\text{for Model \eqref{eq:MT}.}
\end{aligned}
\end{equation*}
Here, $C_{HL}$, $M_{HL}$ and $A_*$ are three positive and explicit constants depending on the matrix $A$, given in Theorem \ref{thm:HL}, and $\Bar{a}=\sup_{i>1}a_i$. In particular, for Model \eqref{eq:MT}, if $a_i=0$ for all $i>1$, the flocking occurs for all initial conditions and for all communication functions. 
\end{thm}

In \cite{shen2008cucker}, the authors show that for Model \eqref{eq:CS} and $\psi(r)=(1+r^2)^{-\frac{\beta}{2}}$ that the flocking is unconditional for $\beta<1$. In \cite{cucker2009critical,dalmao2011cucker}, the authors prove the same result for the discrete-time case and give conditions for $\beta=1$. Applying Theorem \ref{thm:HL_intro} with this communication function, we also show that the flocking is unconditional for $\beta<1$ but we are also able to provide flocking conditions for $\beta\geq1$. Moreover, for Model \eqref{eq:MT} we show that if $\Bar{a}=0$, the flocking phenomenon happens for all communication functions. To our knowledge, no similar flocking results exists. \\

\textbf{Outline:} In Section \ref{sec:preliminaries}, we exhibit the problem in a more formal way and state all our main results. We first set out all the definitions and assumptions we will use in Section \ref{subsec:modelandnot}. Then, we expose the probabilistic interpretation of \eqref{eq:model} in Section \ref{subsec:probint}. In Section \ref{sec:mainresult}, we state all the flocking result for each assumption and for both Models \eqref{eq:CS} and \eqref{eq:MT}, starting with the reversible and irreducible case and the scrambling case, with the hierarchical leadership case afterwards and finishing with the general leadership case. Lastly, in section \ref{sec:proof}, we prove all the previous results. 

\section{Preliminaries}\label{sec:preliminaries}

In this section, we present the model in a more detailed way. We present the probabilistic interpretation and we use it to prove the existence and the uniqueness of the solution of \eqref{eq:model}. We also give a result which links the flocking phenomenon to the convergence in total variation of a certain family of Markov jump processes. Finally, we provide all the results we prove in this article and we compare them to the literature. \\

\subsection{Model and Notations}\label{subsec:modelandnot}

We recall that the model studied in this article is Model \eqref{eq:model} with $Q$ satisfying \eqref{eq:TRM} and \eqref{eq:CS} or \eqref{eq:MT}. In the rest of the paper, we will assume that the interaction function $\psi$ introduced in \eqref{eq:assQ} is positive, decreasing and satisfies $\norm{\psi}_\infty\leq1$ i.e.   

\begin{equation}\label{eq:asspsi}
    \forall (r_1,r_2)\in\mathbb{R}_+^2,\ (r_1-r_2)(\psi(r_1)-\psi(r_2))\leq 0\quad \text{and}\quad\forall r\in\mathbb{R}_+,\ 0<\psi(r)\leq \psi(0)\leq1. 
\end{equation}

Some more general communication functions are also studied in the literature. For instance, in the case of the original Cucker-Smale model, we refer to \cite{park2010cucker,motsch2014heterophilious,jin2018flocking} for a  compactly supported communication function and to \cite{ha2009simple} for a non-bounded one. \\

In all examples, we may assume that
\begin{equation}\label{eq:comini}
    \psi:r\mapsto \frac{1}{(1+r^2)^{\beta/2}},
\end{equation}

where $\beta\geq0$. This is the communication function initially introduced in \cite{cucker2007emergent} and used in most of the articles on this subject.\\

Moreover, we will address Model \eqref{eq:assQ} under the following assumptions: \\

\begin{ass}[Irreducible and reversible assumption]\label{ass:rev}
The graph $\mathcal{G}$ induced by $A$ is irreducible, which means that for all $(i,j)\in\intint{1}{N}^2$, there exists a path from $i$ to $j$. Moreover, there exists a (positive) probability measure $\pi$ on $\intint{1}{N}$ which is reversible for $A$ for all $t\geq0$, i.e. $\forall i\ne j,\quad\pi_iA_{ij}=\pi_jA_{ji}$. \\
\end{ass}

\begin{ass}[Scrambling assumption]\label{ass:scram}
For all $(i,j)\in\intint{1}{N}^2$ such that $i\ne j$, $A_{ij}>0$ or $A_{ji}>0$ or there exists $k\in\intint{1}{N}$ such that $A_{ik}>0$ and $A_{jk}>0$. \\
\end{ass}

\begin{ass}[Hierarchical leadership assumption]\label{ass:HL}
For all $(i,j)\in\intint{1}{N}^2$, $A_{ij}>0$ only if $j<i$ and for all $i>1$ there exists $j<i$ such that $A_{ij}>0$. \\
\end{ass}

\begin{ass}[General leadership assumption]\label{ass:GL}
For all $(i,j)\in\intint{1}{N}^2$ such that $i\ne j$, $i$ leads to $j$ or $j$ leads to $i$ or there exists a vertex $k\in\intint{1}{N}$ such that $i$ leads to $k$ and $j$ leads to $k$. \\
\end{ass}

Let us note that under Assumption \ref{ass:rev}, as $A$ is irreducible, the probability measure $\pi$ is necessarily positive and that we necessarily have $A_{ij}>0$ if and only if $A_{ji}>0$. Moreover, if $Q$ satisfies \eqref{eq:CS} then $\pi$ is reversible for $A$ implies $\pi$ is reversible for $Q_t$ for all $t\geq0$. Under Assumption \ref{ass:HL}, the agent $1$ is autonomous in the sense that for all agent $j\ne1$, $A_{1j}=0$ and then $(x_1(t),v_1(t))=(x_1(0)+v_1(0)t,v_1(0))$. Moreover, Assumptions \ref{ass:rev}, \ref{ass:scram} and \ref{ass:HL} are specific cases of Assumption \ref{ass:GL}. This last assumption is equivalent to $G$ admits a unique closed class. If the unique closed communication class is composed of a single agent, then it is autonomous. It should be noted at this point that an important consequence of \eqref{eq:CS} or \eqref{eq:MT} is that for all $i\ne j$ and $t\geq0$,\ $Q_t(i,j)>0$ if and only if $A_{ij}>0$. Consequently, the interaction graph is also generated by $Q_t$ and hence if $A$ satisfies Assumption \ref{ass:scram}, \ref{ass:HL} or \ref{ass:GL} then $Q_t$ also verifies them. A graphical example of each case discussed in this paper is provided in Figure \ref{fig:graphs}. \\

\begin{figure}[!ht]
    \begin{minipage}[c]{.46\linewidth}
        \centering
        \begin{tikzpicture}[line cap=round,line join=round,>=stealth',scale=0.8]
            \node[draw,circle,color=RED] (1) at (6,7) {1};
            \node[draw,circle,color=RED] (2) at (5,6) {2};
            \node[draw,circle,color=RED] (3) at (7,6) {3};
            \node[draw,circle,color=RED] (4) at (6,5) {4};
            \node[draw,circle,color=RED] (5) at (5,4) {5};
            \node[draw,circle,color=RED] (6) at (6,3) {6};
            \node[draw,circle,color=RED] (7) at (7,4) {7};
            \node[draw,circle,color=RED] (8) at (4,5) {8};
            \node[draw,circle,color=RED] (9) at (8,5) {9};
            \draw[<->] (1) edge (2);
            \draw[<->] (1) edge (3);
            \draw[<->] (2) edge (8);
            \draw[<->] (2) edge (4);
            \draw[<->] (3) edge (4);
            \draw[<->] (3) edge (9);
            \draw[<->] (8) edge (5);
            \draw[<->] (4) edge (5);
            \draw[<->] (4) edge (7);
            \draw[<->] (9) edge (7);
            \draw[<->] (5) edge (6);
            \draw[<->] (7) edge (6);
        \end{tikzpicture}
        \caption*{Irreducible and reversible case.}
        \label{fig:REV}
    \end{minipage} \hfill
    \begin{minipage}[c]{.46\linewidth}
        \centering
        \begin{tikzpicture}[line cap=round,line join=round,>=stealth',scale=0.8]
            \node[draw,circle,color=RED] (1) at (0,0) {1};
            \node[draw,circle,color=RED] (2) at (2,0) {2};
            \node[draw,circle,color=RED] (3) at (3,1.7320508075688774) {3};
            \node[draw,circle,color=BLUE] (4) at (2,3.4641016151377553) {4};
            \node[draw,circle,color=BLUE] (5) at (0,3.4641016151377557) {5};
            \node[draw,circle,color=BLUE] (6) at (-1,1.732050807568879) {6};
            \draw[<->] (1) edge (2);
            \draw[->] (2) edge (3);
            \draw[<->] (3) edge (1);
            \draw[->] (4) edge (3);
            \draw[->] (4) edge (5);
            \draw[->] (5) edge (1);
            \draw[->] (6) edge (1);
            \draw[->] (6) edge (3);
        \end{tikzpicture}
        \caption*{Scrambling case.}
        \label{fig:SCRAM}
    \end{minipage}
    \begin{minipage}[c]{.46\linewidth}
        \centering
        \begin{tikzpicture}[line cap=round,line join=round,>=stealth',scale=0.6]
            \node[draw,circle,color=RED] (1) at (10,5) {1};
            \node[draw,circle,color=BLUE] (2) at (8,5) {2};
            \node[draw,circle,color=BLUE] (3) at (8,7) {3};
            \node[draw,circle,color=BLUE] (4) at (8,3) {4};
            \node[draw,circle,color=BLUE] (5) at (4,3) {5};
            \node[draw,circle,color=BLUE] (6) at (6,4) {6};
            \node[draw,circle,color=BLUE] (7) at (6,6) {7};
            \node[draw,circle,color=BLUE] (8) at (4,5) {8};
            \node[draw,circle,color=BLUE] (9) at (4,7) {9};
            \draw[->] (9) edge (8);
            \draw[->] (9) edge (7);
            \draw[->] (8) edge (6);
            \draw[->] (7) edge (3);
            \draw[->] (7) edge (2);
            \draw[->] (6) edge (2);
            \draw[->] (5) edge (4);
            \draw[->] (4) edge (1);
            \draw[->] (3) edge (2);
            \draw[->] (2) edge (1);
        \end{tikzpicture}
        \caption*{Hierarchical leadership case.}
        \label{fig:HL}
    \end{minipage} \hfill
    \begin{minipage}[c]{.46\linewidth}
        \centering
        \begin{tikzpicture}[line cap=round,line join=round,>=stealth',scale=0.7]
            \node[draw,circle,color=RED] (1) at (0,0) {1};
            \node[draw,circle,color=RED] (2) at (2,0) {2};
            \node[draw,circle,color=RED] (3) at (2,2) {3};
            \node[draw,circle,color=RED] (4) at (0,2) {4};
            \node[draw,circle,color=BLUE] (5) at (2,3.5) {5};
            \node[draw,circle,color=BLUE] (6) at (3.5,2) {6};
            \node[draw,circle,color=BLUE] (7) at (-2,2) {7};
            \node[draw,circle,color=BLUE] (8) at (-2,0) {8};
            \node[draw,circle,color=BLUE] (9) at (-3,1) {9};
            \draw[->] (1) edge (2);
            \draw[->] (2) edge (3);
            \draw[->] (3) edge (4);
            \draw[->] (4) edge (1);
            \draw[->] (1) edge (3);
            \draw[->] (5) edge (3);
            \draw[->] (6) edge (3);
            \draw[<->] (5) edge (6);
            \draw[->] (7) edge (4);
            \draw[->] (7) edge (1);
            \draw[->] (8) edge (1);
            \draw[->] (9) edge (7);
            \draw[->] (9) edge (8);
        \end{tikzpicture}
        \caption*{General leadership case.}
        \label{fig:UCC}
    \end{minipage}
    \caption{Red vertices are those of the unique closed class and the others are plotted in blue.}
    \label{fig:graphs}
\end{figure}

Finally, we define the flocking phenomenon as follows:

\begin{defn}\label{def:flocking}
Let $(x_i(t),v_i(t))_{i\in\intint{1}{N},t\in\mathbb{R}_+}$ be a solution of Equation \eqref{eq:model}. For all $t\geq0$, let $X(t)$ and $V(t)$ be the diameters of positions and velocities at time $t$, defined as 
$$
X(t)=\sup_{i,j}\norm{x_i(t)-x_j(t)}_2\quad\text{and}\quad V(t)=\sup_{i,j}\norm{v_i(t)-v_j(t)}_2.
$$
We say that there is flocking when 
\begin{equation}\label{eq:flocking}
    \left\{\begin{aligned}
                &\sup_{t\geq0}X(t)<+\infty, \\
                &\lim_{t\to+\infty}V(t)=0.
            \end{aligned}\right.
\end{equation}
If \eqref{eq:flocking} is satisfied for all initial conditions then the flocking is unconditional. \\
\end{defn}

\subsection{The probabilistic interpretation}\label{subsec:probint}

Let us assume that $Q_t$ depends on the time $t$ only through the value of the solution at this time, that is for all $i\ne j$ and $t\geq 0$

\begin{equation}\label{eq:defQPsi}
    Q_t(i,j)=\Psi_{ij}((x_1(t),v_1(t)),\dots,(x_N(t),v_N(t)),
\end{equation}

where $\Psi_{ij}$ is an application from $\mathbb{R}^{2dN}$ to $\mathbb{R}_+$. \\

Let $(Y_t)_{t\geq0}$ be a $\intint{1}{N}$-valued time-inhomogeneous Markov jump process of transition rate matrix $(Q_t)_{t\geq0}$. Let $(P_{s,t})_{0\leq s\leq t}$ be its transition function defined as
\begin{equation}\label{eq:defP}
    P_{s,t}(i,j)=\mathbb{P}(Y_t=j\ |\ Y_s=i).
\end{equation}
For all $0\leq s\leq t$, the matrix $P_{s,t}$ is stochastic that is a matrix which verifies \\
\begin{equation}\label{eq:SM}
    \left\{\begin{aligned}
        &\forall (i,j)\in\intint{1}{N}^2,\ P(i,j)\geq0, \\
        &\forall i\in\intint{1}{N},\ \sum_{j=1}^NP(i,j)=1.
    \end{aligned}\right.
\end{equation}
It also satisfies the semi-group property
\begin{equation}\label{eq:semigroup}
    \forall s\leq u\leq t,\ P_{s,t}=P_{s,u}P_{u,t}.
\end{equation}

Moreover, it follows the forward and backward Kolmogorov equations: for all $0\leq s\leq t,$
\begin{equation}\label{eq:kolmo}
    \left\{\begin{aligned}
        & P_{t,t}=I_d, \\
        & \partial_tP_{s,t}=P_{s,t}Q_t, \\
        & \partial_sP_{s,t}=-Q_sP_{s,t}.
    \end{aligned}\right.
\end{equation}

Where $I_d$ is the identity matrix on $\mathbb{R}^d$. See \cite[Theorem 3.1 and 4.1]{feinberg2014solutions} for a proof of this result on a general measurable space (instead of $\intint{1}{N}$). \\

Let $T>0$ be a positive real number and $(Y^{(T)}_t)_{t\in[0,T]}$ be a $\intint{1}{N}$-valued time-inhomogeneous jump process of generator $(\alpha Q_{T-t})_{t\in[0,T]}$. If we set for all $0\leq s\leq t\leq T$ 
$$
p^{(T)}_{s,t}=P^{(T)}_{T-t,T-s},
$$

where $(P^{(T)}_{s,t})_{0\leq s\leq t\leq T}$ is the transition function of $Y^{(T)}$ defined by \eqref{eq:defP}, then $(p^{(T)}_{s,t})_{0\leq s\leq t\leq T}$ is a solution on $[0,T]$ of the following differential equation
\begin{equation}\label{eq:kolmorev}
    \left\{\begin{aligned}
        & p_{t,t}=I_d, \\
        & \partial_tp_{s,t}=\alpha Q_tp_{s,t}, \\
        & \partial_sp_{s,t}=-\alpha p_{s,t}Q_s.
    \end{aligned}\right.
\end{equation}

\begin{thm}\label{thm:probint}
If for all $(i,j)\in\intint{1}{N}^2$ such that $i\ne j,\ \Psi_{ij}$ is non-negative, bounded and locally Lipschits, then for any initial data, the solution of \eqref{eq:model} exists and is unique on $\mathbb{R}_+$ and satisfies for all $i\in\intint{1}{N}$, for all $m\in\intint{1}{d}$ and for all $0\leq t\leq T$, 
\begin{equation*}
    v^m_i(t)=p^{(T)}_{0,t}f_m(i)=\mathbb{E}\left(f_m\left(Y^{(T)}_T\right)\ |\ Y^{(T)}_{T-t}=i\right),
\end{equation*}
where $f_m(i)=v^m_i(0)$ is the $m$-th coordinate of the initial velocity of agent $i$ and $Y^{(T)}$ is a time-inhomogeneous jump process of generator $(\alpha Q_{T-t})_{t\in[0,T]}$. Moreover, we have for all $0\leq s\leq t$, $V(t)\leq V(s)$. 
\end{thm}

\begin{prf}
First, by the Cauchy-Lipschitz Theorem, for any initial condition, there exists a unique maximal solution on $[0,t^*)$. Let $T<t^*$ be a positive real number and $Y^{(T)}$ be the time-inhomogeneous jump process defined as in the statement of Theorem \ref{thm:probint}. Setting $f_m(i)=v^m_i(0)$ the $m$-th coordinate of $v_i(0)$, according to the equation \eqref{eq:kolmorev}, we have:
$$
\partial_tp^{(T)}_{0,t}f_m(i)=\alpha Q_tp^{(T)}_{0,t}f_m(i)=\alpha\sum_{j=1}^NQ_t(i,j)\left(p^{(T)}_{0,t}f_m(j)-p^{(T)}_{0,t}f_m(i)\right),
$$
which is exactly the differential equation followed by $v_i^k:t\mapsto v_i^m(t)$ with the same initial condition. By the uniqueness of the solution on $[0,t^*)$, we conclude that for all $t\leq T<t^*$ and $i\in\intint{1}{N}$, $v_i(t)=p^{(T)}_{0,t}f(i)$. Then, if for all $i\in\intint{1}{N}$ and $m\in\intint{1}{d}$, $f_m(i)\geq0$, we have for all $t\in[0,T]$, $v_i^m(t)\geq0$. By \cite[Theorem 4 (i)]{shen2008cucker}, if $C$ is a closed convex set of $\mathbb{R}^d$, then 
$$
\forall i\in\intint{1}{N},\ v_i(0)\in C\implies \forall t\in[0,T],\ \forall i\in\intint{1}{N},\ v_i(t)\in C.
$$
As the convex envelope of $\{v_i(0)\}_{i\in\intint{1}{N}}$ is a closed bounded convex set containing all the $v_i(0)$, the solution does not blow up. Consequently, we conclude to the existence and the uniqueness of the solution of \eqref{eq:model} on $\mathbb{R}_+$. Furthermore, as the diameter of a set of points is equal to the diameter of its convex envelope, we conclude that $V(t)\leq V(0)$. As $\Psi_{ij}$ does not depend on $t$, $t\mapsto(x_i(s+t),v_i(s+t))_{i\in\intint{1}{N}}$ is a solution \eqref{eq:CS}. Thus, by the uniqueness of the solution, we show that for all $0\leq s\leq t$, $V(t)\leq V(s)$. 
\begin{flushright}
    $\qed$
\end{flushright}
\end{prf}

\begin{rem}
By \cite[Theorem 3.2 and 4.3]{feinberg2014solutions}, the transition function $(P_{s,t})_{0\leq s\leq t}$ defined by \eqref{eq:defP} is the unique solution of \eqref{eq:kolmo} if $\Psi_{ij}$ satisfies assumptions of Theorem \ref{thm:probint}. As we can construct a solution of \eqref{eq:kolmo} from a solution of \eqref{eq:kolmorev}, uniqueness also holds for \eqref{eq:kolmorev}. Thus, given a solution of Equation \eqref{eq:model} on $\mathbb{R}_+$, Equation \eqref{eq:kolmorev} admits a unique global solution $(P^*_{s,t})_{0\leq s\leq t}$ with $Q_t$ given by \eqref{eq:defQPsi} which satisfies
$$
\forall s\leq t,\ P^*_{s,t}=p^{(t)}_{s,t}=P_{0,t-s}^{(t)}\quad\text{and}\quad\forall t\geq0,\ v^m_i(t)=P^*_{0,t}f_m(i).
$$
\end{rem}

In the next corollary, we will highlight the link between the flocking phenomenon and the convergence in total variation of the family of jump processes $(Y_t^{(T)})_{t\leq T}$. Its proof uses a $L^2$-contraction result which goes back to Dobrushin and whose a simple proof could be found in \cite[Lemma 2.2]{dong2019emergent}. \\

\begin{cor}\label{cor:decV}
Let $(x_i(t),v_i(t))_{i\in\intint{1}{N},t\in\mathbb{R}_+}$ be a solution of Equation \eqref{eq:model} and let $(P^*_{s,t})_{0\leq s\leq t}$ be defined as the solution of \eqref{eq:kolmorev} on $\mathbb{R}_+$ for $Q_t$ given by \eqref{eq:defQPsi}. We define $\mu(P^*_{s,t})$ the Dobrushin ergodicity coefficient of $P^*_{s,t}$ as 
$$
\mu(P^*_{s,t})=\inf_{i,j}\sum_{k=1}^NP^*_{s,t}(i,k)\land P^*_{s,t}(j,k),
$$

where $a\land b=\min(a,b)$. Then, for all $0\leq s\leq t\leq T$, we have 
\begin{equation}\label{eq:decV}
    V(t)\leq(1-\mu(P^*_{s,t}))V(s).
\end{equation}
\end{cor}

\begin{prf}
For all $(i,m)\in\intint{1}{N}\times\intint{1}{d}$, we denote $\mathcal{V}_{im}(t)=v_i^m(t)$. By Theorem \ref{thm:probint} and by Equation \eqref{eq:semigroup}, we have for all $0\leq s\leq t\leq T$
$$
\mathcal{V}(t)=P^{(T)}_{T-t,T}\mathcal{V}(0)=P^{(T)}_{T-t,T-s}P^{(T)}_{T-s,T}\mathcal{V}(0)=P^*_{s,t}\mathcal{V}(s).
$$
By \cite[Lemma 2.2]{dong2019emergent}, if $P$ is an $N\times N$ stochastic matrix and $X$ an $N\times d$ matrix, then we have 
$$
\sup_{i,j}\norm{PX(i)-PX(j)}_2\leq (1-\mu(P))\sup_{i,j}\norm{X(i)-X(j)}_2,
$$
where $X(i)$ and $PX(i)$ are the $i$-th row vector of $X$ and $PX$, which allows us to conclude. 
\begin{flushright}
    $\qed$
\end{flushright}
\end{prf}

\section{Main results}\label{sec:mainresult}

In this section, we will state the main results of this article. We will provide flocking conditions for Models \eqref{eq:CS} and \eqref{eq:MT} under Assumptions \ref{ass:rev}, \ref{ass:scram}, \ref{ass:HL} and \ref{ass:GL}. To illustrate these results, we will apply them on particular choices of $A$ and $\psi$. The poofs will be given in section \ref{sec:proof}.\\

\subsection{The irreducible and reversible case}\label{subsec:rev}

For the irreducible and reversible case, the main argument is: if $A$ satisfies Assumption \ref{ass:rev} and $Q_t$ is given by \eqref{eq:CS} then $Q_t$ satisfies Assumption \ref{ass:rev} for all $t\geq0$. As it is not the case for $Q_t$ given by \eqref{eq:MT}, we will only deal with Model \eqref{eq:CS} in this section. 

\begin{defn}
    Let us define for all $h:\intint{1}{N}\to\mathbb{R}$ the two following quantities:

    \begin{equation}\label{eq:EV}
        \left\{\begin{aligned}
            & \mathcal{E}(h)=\frac{1}{2}\sum_{i,j=1}^N(h(i)-h(j))^2\pi_iA_{ij}, \\
            & \mathbb{V}_\pi(h)=\sum_{i=1}^N\pi_i(h(i)-\pi(h))^2.
        \end{aligned}\right.
    \end{equation}
    We say that $A$ satisfies a Poincaré inequality if there exists a constant $c>0$ such that for all $h:\intint{1}{N}\to\mathbb{R}$ satisfying $\mathbb{V}_\pi(h)>0$, we have
    \begin{equation}\label{eq:poincare}
        \mathcal{E}(h)\geq c\mathbb{V}_\pi(h).
    \end{equation}
    In this case, we define the Poincaré constant of $A$ as the highest value which verifies \eqref{eq:poincare} i.e.
    \begin{equation}\label{eq:poincarecst}
        c_P=\inf\left\{\frac{\mathcal{E}(h)}{\mathbb{V}_\pi(h)}\ |\ h:\intint{1}{N}\to\mathbb{R},\ \mathbb{V}_\pi(h)>0\right\}>0.
    \end{equation}
\end{defn}

Let $L=D-A$ be the Laplacian matrix of $A$, where $D_{ij}=\left(\sum_{k=1}^NA_{ik}\right)\delta_{ij}$ with $\delta_{ij}$ the Kronecker symbol. It is known that, using a standard diagonalisation argument, if $A$ satisfies Assumption \ref{ass:rev} then it always satisfies a Poincaré inequality and that its Poincaré constant is the smallest positive eigenvalue of $L$, also known as its Fielder number (see \cite[Theorem 0.31]{chen2004markov}).  

\begin{thm}\label{thm:rev}
If $\psi$ satisfies \eqref{eq:asspsi} and if $A$ satisfies Assumption \ref{ass:rev} then a flocking condition for Model \eqref{eq:CS} is given by:
\begin{equation}\label{eq:rev}
    \widetilde{V}(0)<\frac{\alpha c_P\sqrt{\pi^*}}{2}\int_{X(0)}^{+\infty}\psi(r)\ dr,
\end{equation}
where $\pi^*=\inf_i\ \pi_i$,
$$
\widetilde{V}(0)=\sqrt{\sum_{i=1}^N\pi_i\norm{v_i(0)-v^*}_2^2},
$$
and $v^*=\sum_{i=1}^N\pi_iv_i(0)$. Furthermore, we have for all $i\in\intint{1}{N},\ v_i(t)\underset{t\to+\infty}{\longrightarrow}v^*$.
\end{thm}

\begin{cor}\label{cor:revMF}
If $\alpha=1$ and $A_{ij}=1/N$ then a flocking condition for Model \eqref{eq:CS} is given by:
\begin{equation}\label{eq:revMF}
    \sqrt{\sum_{i=1}^N\norm{v_i(0)-v^*}_2^2}<\frac{1}{2}\int_{X(0)}^{+\infty}\psi(r)\ dr.
\end{equation}
In particular, if $\psi$ is given by \eqref{eq:comini}, the flocking is unconditional when $\beta\leq1$.
\end{cor}

\begin{rem}
In \cite[Theorem 3.2 (ii)]{ha2009simple}, the authors give the following flocking condition,

$$
\sqrt{\sum_{i=1}^N\norm{v_i(0)-v_*}_2^2}<\frac{1}{N}\int_{\widetilde{X}(0)}^{+\infty}\psi(2r)\ dr,
$$

where $\widetilde{X}(0)=\sqrt{\sum_{i=1}^N\norm{x_i(0)-x^*(0)}_2^2}$ where $x^*(t)=\frac{1}{N}\sum_{i=1}^Nx_i(t)$. To compare Condition \ref{eq:revMF} to the previous one, we set $R(t)=\sup_i\norm{x_i(t)-x^*(t)}_2$. As $X(0)\leq2R(0)$, we obtain the following flocking condition
$$
\sqrt{\sum_{i=1}^N\norm{v_i(0)-v_*}_2^2}<\int_{R(0)}^{+\infty}\psi(2r)\ dr.
$$
As $R<\widetilde{X}$, Corollary \ref{cor:revMF} is strictly better than the condition given in \cite{ha2009simple}. \\
\end{rem}

\begin{exmp}
To illustrate the convergence of all the $v_i(t)$ toward $v^*$, we consider the example where for all $i>1$, $A_{ij}=A_i\mathbb{1}_{j=1}$ with $A_i>0$ and we note $B_j=A_{1j}$. \\

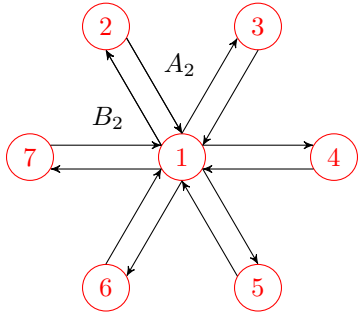
\begin{figure}[!ht]
    \centering
    \begin{tikzpicture}[line cap=round,line join=round,>=stealth',scale=1]
        \node[draw,circle,color=RED] (1) at (1,1.7320508075688776) {1};
        \node[draw,circle,color=RED] (2) at (0,3.4641016151377557) {2};
        \node[draw,circle,color=RED] (3) at (2,3.4641016151377553) {3};
        \node[draw,circle,color=RED] (4) at (3,1.7320508075688774) {4};
        \node[draw,circle,color=RED] (5) at (2,0) {5};
        \node[draw,circle,color=RED] (6) at (0,0) {6};
        \node[draw,circle,color=RED] (7) at (-1,1.732050807568879) {7};
        \draw[->] (1) edge[bend right,draw=none] coordinate[at start](1-b) coordinate[at end](2-b) (2)
            edge[bend left,draw=none] coordinate[at start](1-t) coordinate[at end](2-t) (2)
            (1-t) --node[below left,draw=none, rectangle]{$B_2$} (2-t);
        \draw[->] (2-b) --node[above right,draw=none, rectangle]{$A_2$} (1-b);
        \draw[->] (1) edge[bend right,draw=none] coordinate[at start](1-b) coordinate[at end](3-b) (3)
            edge[bend left,draw=none] coordinate[at start](1-t) coordinate[at end](3-t) (3)
            (1-t) -- (3-t);
        \draw[->] (3-b) -- (1-b);
        \draw[->] (1) edge[bend right,draw=none] coordinate[at start](1-b) coordinate[at end](4-b) (2)
            edge[bend left,draw=none] coordinate[at start](1-t) coordinate[at end](2-t) (2)
            (1-t) -- (2-t);
        \draw[->] (2-b) -- (1-b);
        \draw[->] (1) edge[bend right,draw=none] coordinate[at start](1-b) coordinate[at end](4-b) (4)
            edge[bend left,draw=none] coordinate[at start](1-t) coordinate[at end](4-t) (4)
            (1-t) -- (4-t);
        \draw[->] (4-b) -- (1-b);
        \draw[->] (1) edge[bend right,draw=none] coordinate[at start](1-b) coordinate[at end](5-b) (5)
            edge[bend left,draw=none] coordinate[at start](1-t) coordinate[at end](5-t) (5)
            (1-t) -- (5-t);
        \draw[->] (5-b) -- (1-b);
        \draw[->] (1) edge[bend right,draw=none] coordinate[at start](1-b) coordinate[at end](6-b) (6)
            edge[bend left,draw=none] coordinate[at start](1-t) coordinate[at end](6-t) (6)
            (1-t) -- (6-t);
        \draw[->] (6-b) -- (1-b);
        \draw[->] (1) edge[bend right,draw=none] coordinate[at start](1-b) coordinate[at end](7-b) (7)
            edge[bend left,draw=none] coordinate[at start](1-t) coordinate[at end](7-t) (7)
            (1-t) -- (7-t);
        \draw[->] (7-b) -- (1-b);
    \end{tikzpicture}
    \caption{A star-shaped interaction graph with $N=7$.}
\end{figure}

We can show that $A$ admits as reversible measure the measure defined by:
\begin{equation}\label{eq:revmeasstar}
    \pi_1=\left(1+\sum_{j>1}\frac{B_j}{A_j}\right)^{-1}\ \text{and} \ \forall i>1,\quad \pi_i=\frac{B_i}{A_i}\left(1+\sum_{j>1}\frac{B_j}{A_j}\right)^{-1}.
\end{equation}

By Theorem \ref{prop:limrev}, in case of flocking, the asymptotic speed is given by
\begin{equation}\label{eq:limcelstar}
    v_e=\frac{v_1(0)+\sum_{i>1}\frac{B_i}{A_i}v_i(0)}{1+\sum_{i>1}\frac{B_i}{A_i}}.
\end{equation}

Two simulations of Model \eqref{eq:model} with $Q_t$ given by \eqref{eq:CS} and $\psi$ given by \eqref{eq:comini} for two different interaction matrix $A$ are plotted in figure \ref{fig:starshapeinteraction}.    

\begin{figure}[!ht]
    \begin{minipage}[c]{.46\linewidth}
        \centering
        \includegraphics[scale=0.19]{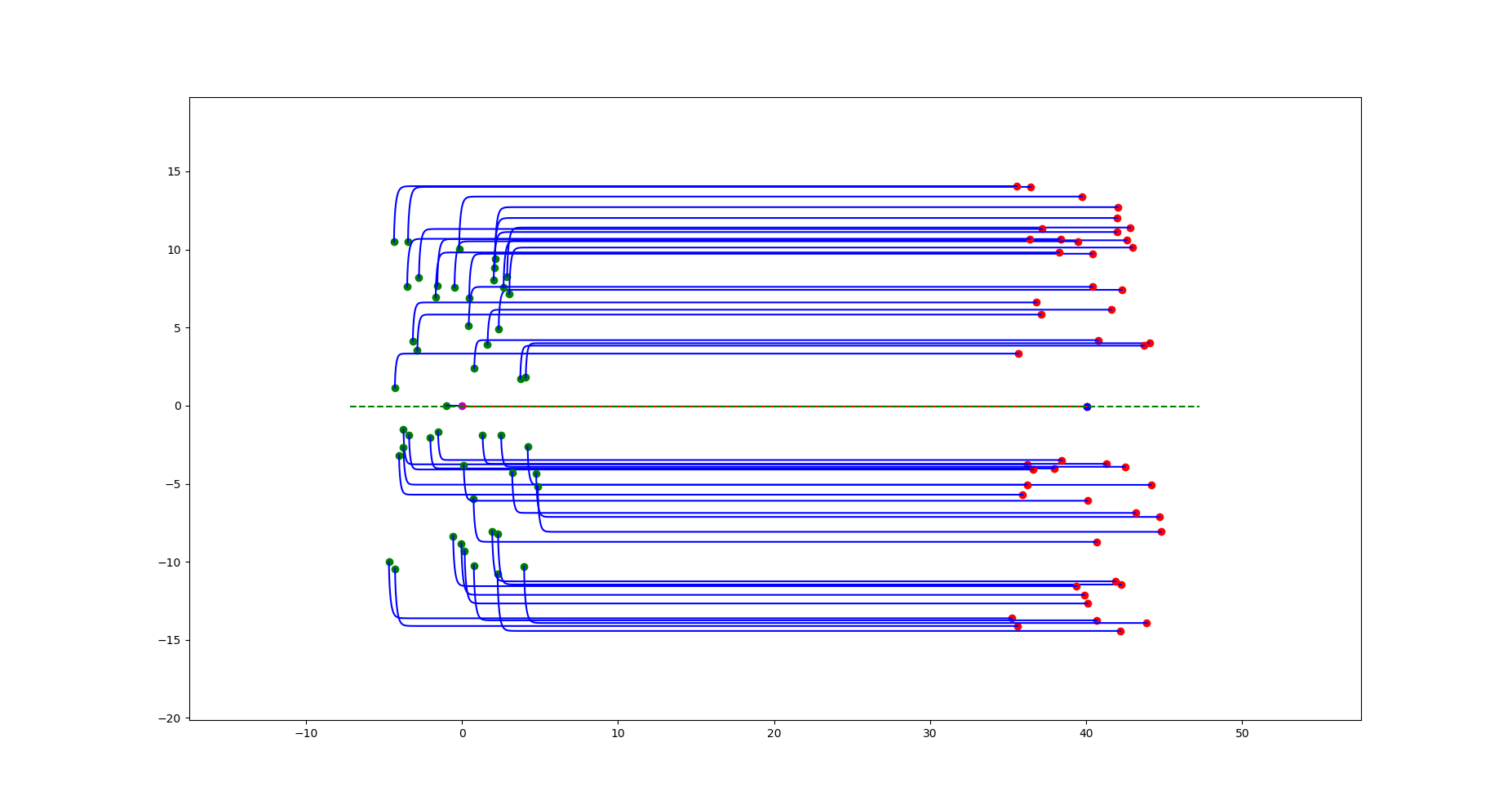}
        \caption*{$A=1$, $B=1$}
    \end{minipage} \hfill
    \begin{minipage}[c]{.46\linewidth}
        \centering
        \includegraphics[scale=0.19]{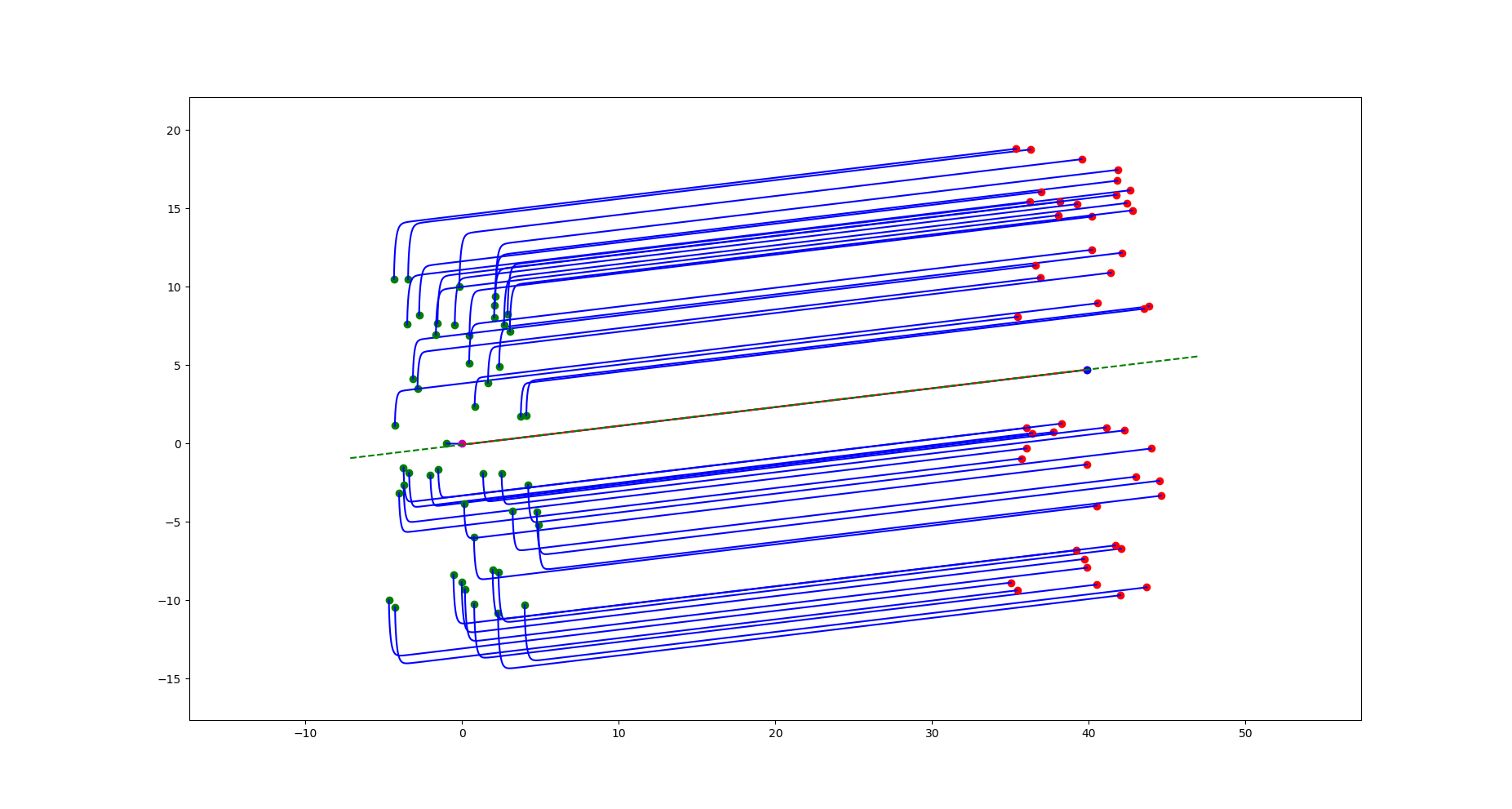}
        \caption*{$A=1$, $(B_i)_{1<i\leq25}=1$, $(B_i)_{25<i\leq49}=1.01$}
    \end{minipage}
    \caption{\small The simulations are made with $N=49$, $\psi$ given by \eqref{eq:comini} with $\alpha=1$ and $\beta=0.9$ (hence the flocking conditions are verified). The trajectories of individuals are represented in blue while the direction of the computed asymptotic speed is represented by a green dotted line. The individual $1$ is initially located at $(0,0)$ with a speed of $(1,0)$. The others are divided into two groups. The first one (resp. the second one) is composed by individuals from $2$ to $25$ (resp. from $26$ to $49$) whose the positions are independently and uniformly drawn in $[-5,5]\times[-1,-11]$ (resp. in $[-5,5]\times[1,11]$) with a speed of $(0,-1)$ (resp. of $(0,1)$). Moreover, we assume that for all $i>1$, $A_i=1$.}
    \label{fig:starshapeinteraction}
\end{figure}
\end{exmp}

\subsection{The Scrambling case}\label{subsec:sram}

Hereafter, we will denote by $B$ the matrix defined by:

\begin{equation}\label{eq:matMT}
    B_{ij}=\frac{A_{ij}}{a_i+\sum_{k\ne i}A_{ik}}.
\end{equation}

\begin{thm}\label{thm:scram}
If $\psi$ satisfies \eqref{eq:asspsi} and if $A$ satisfies Assumption \ref{ass:scram} then a flocking condition is given by:

\begin{alignat}{2}
    &V(0)<\alpha\chi(A)\int^{+\infty}_{X(0)}\psi(r)\ dr,\quad &\text{for Model \eqref{eq:CS},} \label{eq:scramCS} \\
    &V(0)<\alpha\chi(B)\int^{+\infty}_{X(0)}\psi(r)\ dr,\quad &\text{for Model \eqref{eq:MT},} \label{eq:scramMT}
\end{alignat}
where 
$$
\chi(A)=\min_{i\ne j}\left(A_{ij}+A_{ji}+\sum_{k\ne i,j}A_{ik}\land A_{jk}\right).
$$
and $B$ is defined in \eqref{eq:matMT}.
\end{thm}

Replacing $A$ by $B$ in \eqref{eq:scramCS}, the flocking conditions obtained in Theorem \ref{thm:scram} for Models \ref{eq:CS} and \ref{eq:MT} are identical. This is due to the fact that, in the proof, we lower bound $Q_t(i,j)$ by $B_{ij}\psi(X(t))$, which is equivalent to comparing it to \eqref{eq:CS}.

\begin{cor}
If $\alpha=1$, $A_{ij}=1/N$ and for Model \eqref{eq:MT} $a_i=1/N$ a flocking condition for both Models \eqref{eq:CS} and \eqref{eq:MT} is given by:
\begin{equation}
    V(0)<\int_{X(0)}^{+\infty}\psi(r)\ dr.
\end{equation}
In particular, the flocking is unconditional when $\beta\leq1$.
\end{cor}

\begin{rem}
This condition was given in \cite[Theorem 2.2]{carrillo2017review} for Model \eqref{eq:CS} and in \cite[Theorem 2.1]{choi2016cucker} for Model \eqref{eq:MT}. In contrast to \eqref{eq:revMF}, this condition does not vanish when $N$ becomes large because the quantity $V(0)$ does not depend explicitly on $N$. For instance, if the initial velocities are uniformly and independently drawn from the set $\{v\in\mathbb{R}^d\ |\ \norm{v}_2\leq1\}$, we have $V(0)\leq2$ and $\sqrt{\sum_{i=1}^N\norm{v_i(0)-v_*}_2^2}\underset{N\to+\infty}{\longrightarrow}+\infty$ by the law of large numbers. \\
\end{rem}

\subsection{The Hierarchical leadership case}\label{subsec:HL}

In this section, we state a more detailed version of Theorem \ref{thm:HL_intro}. \\ 

Let $L_i$ be the set of paths from $i$ to $1$. Let $\abs{l}$ be the length of the path $l$ which is the number of edges it contains. For all $i>1$, let $h_i=\sup\{\abs{l}\ |\ l\in L_i\}$ be the height of node $i$. Finally, let $H=\sup_{i>1}h_i$ be the height of the interaction graph. \\

\begin{thm}\label{thm:HL}
If $\psi$ satisfies \eqref{eq:asspsi} and if $A$ satisfies Assumption \ref{ass:HL} then a flocking condition is given by:
\begin{alignat}{2}
    &V(0)<C_{HL}\sup_{r\geq X(0)}(r-X(0))\psi(r),\quad &\text{for Model \eqref{eq:CS},} \label{eq:HLCS} \\
    &V(0)<M_{HL}\sup_{r\geq X(0)}(r-X(0))\frac{(\Bar{a}+A_*)\psi(r)}{\Bar{a}+A_*\psi(r)},\quad &\text{for Model \eqref{eq:MT},} \label{eq:HLMT}
\end{alignat}
where
\begin{equation*}
    \begin{aligned}
        &C_{HL}=\frac{\alpha A_*}{H},\quad A_*=\inf_{i>1}\sum_{j\ne i}A_{ij}, \\
        &M_{HL}=\frac{\alpha B_*}{H},\quad B_*=\inf_{i>1}\sum_{j\ne i}B_{ij}
    \end{aligned}
\end{equation*}
and $B$ given by \eqref{eq:matMT}. In particular, if $\Bar{a}=0$ the flocking is unconditional. \\
\end{thm}

As an example, we apply Theorem \ref{thm:HL} to $\psi$ given by \eqref{eq:comini}, which is the communication function used in \cite{shen2008cucker}.

\begin{cor}\label{cor:HLcomini}
If $A$ satisfies Assumption \ref{ass:HL} and if $\psi$ is given by \eqref{eq:comini} then the flocking phenomenon occurs for Model \eqref{eq:CS} if
\begin{itemize}
    \item $\beta<1$ (the flocking phenomenon is therefore unconditional).
    \item $\beta=1$ and 
    \begin{equation}\label{eq:HLCScomini=1}
        V(0)<C_{HL},
    \end{equation}
    where $C_{HL}$ is defined in Theorem \ref{thm:HL}.
    \item $\beta>1$ and
    \begin{equation}\label{eq:HLCScomini>1}
        V(0)<C_{HL}\frac{r^*-X(0)}{(1+{r^*}^2)^{\beta/2}},
    \end{equation}
    where
    $$
    r^*=\frac{\sqrt{(\beta X(0))^2+4(\beta-1)}+\beta X(0)}{2(\beta-1)}.
    $$
\end{itemize}
For Model \eqref{eq:MT}, the flocking phenomenon occurs if

\begin{itemize}
    \item $\Bar{a}=0$ or $\beta<1$ (the flocking phenomenon is therefore unconditional).
    \item $\Bar{a}>0$, $\beta=1$ and 
    \begin{equation}\label{eq:HLMTcomini=1}
        C(0)<M_{HL}\frac{\Bar{a}+A_*}{\Bar{a}},
    \end{equation}
    where $M_{HL}$, $A_*$ and $\Bar{a}$ are defined in Theorem \ref{thm:HL}.
    \item $\Bar{a}>0$, $\beta>1$ and
    \begin{equation}\label{eq:HLMTcomini>1}
        V(0)<M_{HL}\frac{(\Bar{a}+A_*)(r^*-X(0))}{A_*+\Bar{a}(1+{r^*}^2)^{\beta/2}},
    \end{equation}
    where $r^*$ is the solution on $[X(0),+\infty)$ of
    \begin{equation}\label{eq:HLMTrstar}
        (1-\beta)r^2+\beta X(0)r+1=-\frac{A_*}{\Bar{a}}(1+r^2)^{1-\beta/2}.
    \end{equation}
\end{itemize}
\end{cor}

Let us mention that \eqref{eq:HLMTrstar} admits a unique solution on $[X(0),+\infty)$ as $\beta>1$ but we cannot find a closed-form solution in the general case. Nevertheless, if $\beta=2$ then we have 
$$
r^*=X(0)+\sqrt{X(0)^2+1+\frac{A_*}{\Bar{a}}}.
$$

\subsection{The General leadership case}\label{subsec:GL}

It is easy to see that Assumption \ref{ass:GL} is the minimal assumption to hope to observe the flocking phenomenon. Intuitively, on the one hand there necessarily exists at least one closed communication class. On the other hand, if there exists more than one, the velocities of individuals of the two classes are not able to equalise because they do not see each other. \\ 

Let $L$ be the set of paths in the interaction graph $\mathcal{G}$ and $L_{ij}=\{(l^i,l^j)\in L^2\ |\ l^i_0=i,\ l^j_0=j,\ l^i_{\abs{l^i}}=l^j_{\abs{l^j}}\}$ the set of coalescence paths of $i$ and $j$. If there exist a path $l$ from $i$ to $j$ (or from $j$ to $i$) we assume that $(l,\{j\})$ belongs to $L_{ij}$ where $\{j\}$ is a path of length $0$. Let $d_{ij}=\inf\{\max(\abs{l^i},\abs{l^j})\ |\ (l^i,l^j)\in L_{ij}\}$ be the coalescence distance of $i$ and $j$ and $D=\sup_{i\ne j}d_{ij}$ the coalescence diameter of $\mathcal{G}$. We remark that Assumption \ref{ass:GL} holds if and only if $D<+\infty$. \\

\begin{thm}\label{thm:GL}
If $\psi$ satisfies \eqref{eq:asspsi} and if $A$ satisfies Assumption \ref{ass:GL} then a flocking condition is given by:

\begin{alignat}{2}
    &V(0)<C_{GL}\sup_{r\geq X(0)}(r-X(0))\psi(r)^D,\quad &\text{for Model \eqref{eq:CS},} \label{eq:GLCS} \\
    &V(0)<M_{GL}\sup_{r\geq X(0)}(r-X(0))\left(\frac{\psi(r)}{K+\hat{A}\psi(r)}\right)^D,\quad &\text{for Model \eqref{eq:MT},} \label{eq:GLMT}
\end{alignat}
where $K=\inf\{a_i+\sum_{k\ne i,j} A_{ik}\ |\ i\ne j, A_{ij}>0\}$ and
\begin{equation*}
\begin{aligned}
    &C_{GL}=\alpha\left(\frac{\hat{A}}{D}\right)^D\left(\frac{D-1}{\Bar{A}}\right)^{D-1}e^{1-D},\ \hat{A}=\inf\{A_{ij}\ |\ i\ne j,A_{ij}>0\},\ \Bar{A}=\sup_i\sum_{j\ne i}A_{ij}, \\
    &M_{GL}=\alpha\left(\frac{\hat{A}}{D}\right)^D\left(\frac{D-1}{\Bar{B}}\right)^{D-1}e^{1-D},\ \Bar{B}=\sup_i\sum_{j\ne i}B_{ij}
\end{aligned}
\end{equation*}
and $B$ is given by \eqref{eq:matMT}. If $D=1$, we set $0^0=1$ for sake of presentation. \\
\end{thm}

If the matrix $A$ given by $A_{ij}=1/N$ and $\alpha=1$, we can easily see that Condition \eqref{eq:scramCS} is sharper that \eqref{eq:GLCS}. This is due in particular to the fact that, as $\psi$ is positive and decreasing,
$$
\sup_{x\geq X(0)}(x-X(0))\psi(x)\leq \int_{X(0)}^{+\infty}\psi(s)\ ds.
$$

\begin{cor}\label{cor:GLcomini}
If $\psi$ is given by \eqref{eq:comini} and if $A$ satisfies Assumption \ref{ass:GL} then the flocking phenomenon occurs for Model \eqref{eq:CS} if
\begin{itemize}
    \item $\beta<1/D$, the flocking phenomenon is therefore unconditional.
    \item $\beta=1/D$ and 
    \begin{equation}\label{eq:GLCScomini=1}
        V(0)<C_{GL},
    \end{equation}
    where $C_{GL}$ is defined in Theorem \ref{thm:GL}.
    \item $\beta>1/D$ and 
    \begin{equation}\label{eq:GLCScomini>1}
        V(0)<C_{GL}\frac{r^*-X(0)}{(1+{r^*}^2)^{\frac{\beta D}{2}}},
    \end{equation}
    where
    $$
    r^*=\frac{\sqrt{(\beta DX(0))^2+4(\beta D-1)}+\beta D X(0)}{2(\beta D-1)}.
    $$
\end{itemize}

For Model \eqref{eq:MT}, the flocking phenomenon occurs if

\begin{itemize}
    \item $\beta<1/D$, the flocking phenomenon is therefore unconditional.
    \item $\beta=1/D$ and 
    \begin{equation}\label{eq:GLMTcomini=1}
        V(0)<\frac{M_{GL}}{K},
    \end{equation}
    where $M_{GL}$ and $K$ are defined in Theorem \ref{thm:GL}.
    \item $\beta>1/D$ and 
    \begin{equation}\label{eq:GLMTcomini>1}
        V(0)<M_{GL}\frac{r^*-X(0)}{(\hat{A}+K(1+{r^*}^2))^{\frac{\beta D}{2}}},
    \end{equation}
    where $r^*$ is the solution on $[X(0),+\infty)$ of
    \begin{equation}\label{eq:GLMTrstar}
        (1-\beta D)r^2+\beta D X(0)r+1=-\frac{\hat{A}}{K}(1+r^2)^{1-\frac{\beta D}{2}}.
    \end{equation}
\end{itemize}
\end{cor}

As before, Equation \eqref{eq:GLMTrstar} does not admit a closed-form solution in the general case but for $\beta D=2$ we have
$$
r^*=X(0)+\sqrt{X(0)^2+1+\frac{\hat{A}}{K}}.
$$

\begin{rem}
In \cite{dong2016flocking}, the authors address the same model as Corollary \ref{cor:GLcomini} with the particular case where $A\in\{0,1\}^{N\times N}$. Setting $\Bar{h}$ the minimum height of a directed spanning tree of the transpose of the interaction graph $\mathcal{G}$, they show that if $\beta<1/\Bar{h}$, the flocking phenomenon is unconditional. As any directed spanning tree of the transpose of $\mathcal{G}$ gives a set of coalescence paths for each pair of vertices, it is easy to show that $D\leq\Bar{h}$. Moreover, taking $A_{ij}=1$ if $j=i+1$ modulo $N$ and $0$ otherwise, we have $\Bar{h}=N-1$ and $D=\left\lfloor\frac{N}{2}\right\rfloor$ which shows that the inequality can be strict and that $\Bar{h}-D$ is unbounded. \\  

Let us assume that $\Bar{h}=D$. Then, for $\beta=1/D$, the authors give the following condition (see \cite[Theorem 4]{dong2016flocking}):
$$
V(0)<\alpha\frac{n_rD^{D-1}}{D!}e^{-\Bar{n}D},
$$
with $n_r$ the number of edges on the unique closed class and $\Bar{n}$ is the maximal degree of an edge. Moreover , we have $\hat{A}=1$ and $\Bar{A}=\Bar{n}$. Consequently, Theorem \ref{thm:HL} gives the following condition,
$$
V(0)<\alpha\frac{(D-1)^{D-1}}{D^D}(\Bar{n}e)^{1-D}.
$$
Comparing these two conditions in the limit of large $D$, we can show that Condition \eqref{eq:GLCScomini=1} is sharper when
$$
\frac{n_r}{\Bar{n}\sqrt{2\pi D}}e^{(2-\Bar{n}+\ln(\Bar{n}))D}\leq 1,
$$
which implies that Condition \eqref{eq:GLCScomini=1} has a better behaviour in $D$ when $\Bar{n}\geq4$. \\
\end{rem}

\begin{rem}
Corollary \ref{cor:HLcomini} gives a much sharper flocking condition than Corollary \ref{cor:GLcomini}. Firstly because the unconditional flocking is ensured for larger values of $\beta$. Secondly because 
$$
C_{HL}\leq C_{GL} \iff H\geq \frac{A_*}{\Bar{A}}\left(\frac{\Bar{A}}{\hat{A}}\right)^D\frac{D^D}{(D-1)^{D-1}}e^{D-1}\underset{D\to+\infty}{\sim}\frac{A_*D}{\Bar{A}}\left(\frac{e\Bar{A}}{\hat{A}}\right)^D.
$$

Which shows that $H$ must be much larger than $D$ to have $C_{GL}$ larger than $C_{HL}$. However, let us mention that $H-D\leq N-2$ and that for all $N>1$, there exists a graph such that this bound is reached (consider the matrix $A=(A_{ij})_{1\leq i,j\leq N}$ with $A_{ij}=1$ if $j=1$ or $j=i-1$ and $A_{ij}=0$ otherwise). Finally, if $\beta>1$ and $C_M=C_{HL}$, the condition \eqref{eq:HLCScomini>1} remains sharper than \eqref{eq:GLCScomini>1} as the right-hand side of \eqref{eq:HLCScomini>1} is decreasing in $\beta$ for $\beta>1$ and that $D\geq1$. \\
\end{rem}

\section{Proof of main results}\label{sec:proof}

In this section, we first provide sufficient conditions on the solutions to ensure that the flocking phenomenon happens. Following the order of the previous section, we use these results to prove Theorems \ref{thm:rev}, \ref{thm:scram}, \ref{thm:HL} and \ref{thm:GL}. 

\subsection{Sufficient conditions for the Flocking phenomenon}\label{subsec:sufcond}

Let us explore three different methods to find flocking conditions. These methods can be used for Models \eqref{eq:CS} and \eqref{eq:MT} as they just highlight conditions on $(x_i,v_i)_{i\in\intint{1}{N}}$ which ensure that \eqref{eq:flocking} is satisfied. \\

\begin{defn} 
Let $\widetilde{X}$ and $\widetilde{V}$ be two functions from $\mathbb{R}_+$ to $\mathbb{R}_+$. We say that the couple $(\widetilde{X},\widetilde{V})$ satisfies a System of Dissipative Differential Inequalities (SDDI) if $\widetilde{X}$ and $\widetilde{V}$ are continuous and piecewise continuously differentiable and if there exists $\phi:\mathbb{R}_+\to\mathbb{R}_+$ such that for all $t\geq0$ such that $\widetilde{X}$ and $\widetilde{V}$ are differentiable,
\begin{equation}\label{eq:SDDI}
    \left\{\begin{aligned}
        &\frac{d\widetilde{X}}{dt}(t)\leq \widetilde{V}(t), \\
        &\frac{d\widetilde{V}}{dt}(t)\leq -\phi(\widetilde{X}(t))\widetilde{V}(t).
    \end{aligned}\right.
\end{equation}
\end{defn}

One of the most common techniques used to establish flocking conditions is to find two quantities $\widetilde{X}$ and $\widetilde{V}$ which verify a SDDI and such that
\begin{equation*}
    X\leq \widetilde{X}\quad\text{and}\quad V\leq \widetilde{V}.
\end{equation*}

The main property of SDDI we will use in this article is the following. It may be seen as a generalisation of the Grönwall lemma. The proof of the following proposition can be found in \cite[Theorem 3.2 (ii)]{ha2009simple} in the case where $\widetilde{X}$ and $\widetilde{V}$ are continuously differentiable. The generalisation to the piecewise continuously differentiable is straightforward by applying \cite[Theorem 3.2 (ii)]{ha2009simple} on each interval on which $\widetilde{X}$ and $\widetilde{V}$ are continuously differentiable. \\

\begin{prop}\label{prop:SDDI}
Let $X$ and $V$ be two continuous and piecewise continuously differentiable functions which satisfy \eqref{eq:SDDI} for a given positive and decreasing function $\phi$. Let us assume that
\begin{equation*}
    V(0)<\int_{X(0)}^{+\infty}\phi(r)\ dr.  
\end{equation*}
Then setting $X_M>0$ the quantity which verifies 
$$
V(0)=\int_{X(0)}^{X_M}\phi(r)\ dr,
$$
we have for all $t\geq0$, 
$$
X(t)\leq X_M,\ V(t)\leq V(0)e^{-\phi(X_M)t}.
$$
\end{prop}

\begin{prop}\label{prop:HL}
Let us suppose that there exists $C:\mathbb{R}_+^2\mapsto\mathbb{R}_+$ such that

\begin{equation}\label{eq:propHL}
    \left\{\begin{aligned}
        &\forall t\geq0,\ r\mapsto C(t,r)\ \text{is increasing}, \\
        &\forall r\geq X(0),\ C(t,r)\underset{t\to+\infty}{\longrightarrow} 0, \\
        &\forall t\geq0,\ 1-\mu(P^*_{0,t})\leq C(t,\sup_{s\leq t}X(s)),
    \end{aligned}\right.
\end{equation}
where $\mu(P^*_{0,t})$ is defined in Corollary \ref{cor:decV}. Then finding $r_0\geq X(0)$ such that
\begin{equation}\label{eq:assHL}
    r_0-X(0)>V(0)\int_0^{+\infty}C(s,r_0)\ ds,
\end{equation}
is a flocking condition.
\end{prop}

\begin{prf}
As $((x_i,v_i))_{i\in\intint{1}{N}}$ is solution of Equation \eqref{eq:model},
$$
\abs{x_i(t)-x_j(t)}=\abs{x_i(0)-x_j(0)+\int_0^t(v_i(s)-v_j(s))ds}\leq\abs{x_i(0)-x_j(0)}+\int_0^t\abs{v_i(s)-v_j(s)}ds,
$$
which leads to
$$
X(t)\leq X(0)+\int_0^tV(s)ds.
$$
Then let $r_0\geq X(0)$ be a positive real number satisfying \eqref{eq:assHL} and 
$$
\tau=\sup\left\{ t\geq0\ |\ \sup_{s\leq t} X(s)\leq r_0\right\}.
$$
Let us assume that $\tau<+\infty$. By continuity of $X$, we have $\sup_{s\leq \tau} X(s)=X(\tau)=r_0$ and then,
$$
r_0-X(0)\leq\int_0^{\tau}V(s)\ ds.
$$
Using Corollary \ref{cor:decV} and Assumptions \eqref{eq:propHL} we have 
$$
r_0-X(0)\leq V(0)\int_0^{+\infty}C(s,r_0)\ ds.
$$
Which is impossible by \eqref{eq:assHL}. By contradiction, we necessary have $\tau=+\infty$ and so, for all $t\geq0$, $X(t)\leq r_0$. As $C(t,.)$ is increasing, we have $V(t)\leq V(0)C(t,r_0)$. As $C(.,r_0)$ is vanishing, Condition \eqref{eq:flocking} is verified, which allows us to conclude.
\begin{flushright}
    $\qed$
\end{flushright}
\end{prf}

\begin{prop}\label{prop:GL}
We assume that there exists $C:\mathbb{R}^2_+\mapsto\mathbb{R}_+$ such as 

\begin{equation}\label{eq:propGL}
    \left\{\begin{aligned}
        &\forall t\geq0,\ r\mapsto C(t,r)\ \text{is decreasing}, \\
        &\forall s\leq t,\ \mu(P^*_{s,t})\geq C(t-s,\sup_{u\leq t}X(u)).
    \end{aligned}\right.
\end{equation}

Then finding $r_0\geq X(0)$ and $t_0>0$ such that

\begin{equation}\label{eq:assGL}
    r_0-X(0)>V(0)\frac{t_0}{C(t_0,r_0)},
\end{equation}
is a flocking condition.
\end{prop}

\begin{prf}
Using \eqref{eq:decV} and \eqref{eq:propGL}, for all $t>0$ and $n\in\mathbb{N}$, if $r\geq\sup_{s\leq nt}X(s)$, we have 
\begin{equation}\label{eq:prfpropGL}
    V(nt)\leq(1-C(t,r))^n V(0).
\end{equation}
Let $r_0\geq X(0)$ and $t_0>0$ be two real numbers satisfying \eqref{eq:assGL}. Let us define $\tau$ as in proof of Proposition \ref{prop:HL}. If $\tau<+\infty$, we have 
$$
r_0-X(0)\leq \int_0^{\tau}V(s)\ ds.
$$
Moreover, as $V$ is decreasing (Theorem \ref{thm:probint}), we have for all integer $k\geq0$ 
$$
\int_{kt_0}^{(k+1)t_0}V(s)\ ds\leq V(kt_0)t_0.
$$
Then, setting $n_0=\lceil \tau/t_0\rceil$, we have, by definition of $\tau$, $r_0\geq\sup_{s\leq(n_0-1)t_0}X(s)$ and so
$$
r_0-X(0)\leq V(0)t_0\sum_{k=0}^{n_0-1}(1-C(t_0,r_0))^k\leq V(0)\frac{t_0}{C(t_0,r_0)}.
$$
The same argument used at the end of the proof of Proposition \ref{prop:HL} allows us to conclude that if \eqref{eq:assGL} is verified, $\sup_{t\geq0}X(t)\leq r_0$ and so, by \eqref{eq:prfpropGL}, $V(nt_0)\underset{n\to+\infty}{\longrightarrow}0$. As $V$ is decreasing, we have $V(t)\underset{t\to+\infty}{\longrightarrow}0$.
\begin{flushright}
    $\qed$
\end{flushright}
\end{prf}  

\subsection{The irreducible and reversible case}\label{subsec:prfrev}

Here, we examine the case where the matrix $A$ satisfies Assumption \ref{ass:rev}. Let $\pi$ be its reversible probability measure. We show in this section that in this case, we can find a closed-form expression of the asymptotic speed and a flocking condition using a Poincaré-type inequality. \\

\begin{prop}\label{prop:limrev}
If $\pi$ is a reversible probability measure for $A$ and if $(x_i,v_i)_{i\in\intint{1}{N}}$ is a solution of \eqref{eq:CS} which satisfies the flocking conditions \eqref{eq:flocking}, then
$$
\norm{v_i(t)-v^*}_2\underset{t\to+\infty}{\longrightarrow}0.
$$
where $v^*=\sum^N_{i=1}\pi_iv_i(0)$. 
\end{prop}

\begin{prf}
The proof of this result directly comes from the fact that the flocking conditions \eqref{eq:flocking} are equivalent to 
\begin{align}
    &\forall i\in\intint{1}{N},\ \underset{t\geq0}{\sup}\norm{x_i(t)-x^*(t)}_2<+\infty, \\
    &\forall i\in\intint{1}{N},\ \underset{t\to+\infty}{\lim}\norm{v_i(t)-v^*(t)}_2=0,
\end{align}
where $x^*(t)=\sum^N_{i=1}\pi_ix_i(t)$ and $v^*(t)=\sum^N_{i=1}\pi_iv_i(t)$ and the fact that
$$
\frac{dv^*}{dt}(t)=\sum_{i>1}\sum_{j<i}(\pi_iA_{ij}-\pi_jA_{ji})\psi\left(\norm{x_i(t)-x_j(t)}_2\right)(v_j(t)-v_i(t))=0.
$$
\begin{flushright}
    $\qed$
\end{flushright}
\end{prf}

\begin{rem}
In Proposition \ref{prop:limrev}, we did not use the fact that $A$ is irreducible. As a result, it can be applied on a matrix $A$ which satisfies Assumption \ref{ass:GL} and whose restriction to the unique closed class admits a reversible positive probability measure. \\ 
\end{rem}

\begin{proof}[Proof of Theorem  \ref{thm:rev}]
Let $f(i)=(f_1(i),\dots,f_d(i))=v_i(0)$ be the initial velocity of agent $i$, we will show that the quantities 
$$
\left\{\begin{aligned}
    & X(t)=\sup_{i,j}\norm{x_i(t)-x_j(t)}_2, \\
    & W(t)=\frac{2}{\sqrt{\pi^*}}\sqrt{\sum_{m=1}^d\mathbb{V}_\pi(P^*_{0,t}f_m)},
\end{aligned}\right.
$$
verify a SDDI. \\

As $\pi$ is reversible for $Q_t$ and as $Q_t$ satisfies \eqref{eq:TRM}, then for all $h:\intint{1}{N}\to\mathbb{R}$, we have $\pi(Q_th)=0$. Thus, we have for all $h:\intint{1}{N}\to\mathbb{R}$,
$$
\frac{d}{dt}\pi(P^*_{0,t}h)=\pi(Q_tP^*_{0,t}h)=0.
$$
As a result, for all $t\geq0$, $\pi(P^*_{0,t}f)=\pi(f)=v^*$ where $f(i)=v_i(0)$ and $v^*=\sum^N_{i=1}\pi_iv_i(0)$. As $\pi$ is positive, it follows that
$$
\sup_{i,j}\norm{v_i(t)-v_j(t)}_2^2\leq4\sum_{i=1}^N\norm{v_i(t)-v^*}_2^2\leq \frac{4}{\pi^*}\sum_{m=1}^d\sum_{i=1}^N\pi_i(P^*_{0,t}f_m(i)-\pi(f_m))^2 
=\frac{4}{\pi^*}\sum_{m=1}^d\mathbb{V}_\pi(P^*_{0,t}f_m),
$$
which shows that $V\leq W$. \\

Consequently, on the one hand, by definition of $X$, we have for all $t\geq 0$ such that $X$ and $V$ are differentiable, 
$$
\abs{\frac{dX^2}{dt}(t)}\leq 2X(t)V(t)\leq 2X(t)W(t),
$$

which proves the first inequality of \eqref{eq:SDDI}. On the other hand, we have
$$
\frac{d}{dt}\mathbb{V}_\pi(P^*_{0,t}f_m)=2\alpha\sum_{i=1}^N\pi_i(P^*_{0,t}f_m(i)-\pi(f_m))Q_tP^*_{0,t}f_m(i)=2\alpha\scalar{P^*_{0,t}f_m}{Q_tP^*_{0,t}f_m}_\pi=-2\alpha\mathcal{E}_t(P^*_{0,t}f_m),
$$
where
$$
\mathcal{E}_t(h):=\frac{1}{2}\sum_{i,j=1}^N(h(i)-h(j))^2\pi_iQ_t(i,j).
$$

Finally, by definition of $Q_t$ and as $\psi$ is decreasing, for all $h:\intint{1}{N}\to\mathbb{R}$,
$$
\mathcal{E}_t(h)\geq\mathcal{E}(h)\psi(X(t)),
$$
where $\mathcal{E}(h)$ is defined in \eqref{eq:EV}. Using the Poincaré inequality \eqref{eq:poincare} verified by the matrix $A$ with its optimal constant $c_P$ \eqref{eq:poincarecst}, we directly have 
$$
\frac{d}{dt}\mathbb{V}_\pi(P^*_{0,t}f_m)\leq-2\alpha c_P\psi(X(t))\mathbb{V}_\pi(P^*_{0,t}f_m),
$$
which leads to the conclusion that $(X,W)$ verify 
$$
\left\{\begin{aligned}
&\abs{\frac{dX}{dt}}\leq V, \\
&\frac{dW}{dt}\leq -\alpha c_P\psi(X)W.
\end{aligned}\right.
$$
Using Proposition \ref{prop:SDDI} and the expression of $\mathbb{V}_\pi(f)$ given by \eqref{eq:EV}, we conclude to the expected result.
\begin{flushright}
    $\qed$
\end{flushright}
\end{proof}

\subsection{The Scrambling case}\label{subsec:prfscram}

Here we show that the probabilistic interpretation allows us to obtain a very simple proof in the case where $A$ satisfies Assumption \ref{ass:scram}. \\

\begin{proof}[Proof of Theorem  \ref{thm:scram}]
We recall that, as in the proof of Theorem \ref{thm:rev}, 
$$
\abs{\frac{dX}{dt}}\leq V.
$$
By \eqref{eq:decV}, we have for all $t\geq0$ and $\varepsilon\geq0$ 
$$
\frac{V(t+\varepsilon)-V(t)}{\varepsilon}\leq -\frac{\mu(P^*_{t,t+\varepsilon})}{\varepsilon}V(t).
$$
In addition, using \eqref{eq:kolmorev} and the fact that $P^*_{t,t}=I_d$, we have 
$$
\begin{aligned}
\frac{\mu(P^*_{t,t+\varepsilon})}{\varepsilon}=&\min_{i,j}\sum_{k=1}^N\frac{P^*_{t,t+\varepsilon}(i,k)}{\varepsilon}\land\frac{P^*_{t,t+\varepsilon}(j,k)}{\varepsilon} \\
&=\min_{i,j}\sum_{k=1}^N\left(\frac{P^*_{t,t+\varepsilon}(i,k)-\delta_{ik}}{\varepsilon}+\frac{\delta_{ik}}{\varepsilon}\right)\land\left(\frac{P^*_{t,t+\varepsilon}(j,k)-\delta_{jk}}{\varepsilon}+\frac{\delta_{jk}}{\varepsilon}\right),
\end{aligned}
$$
where $\delta_{ij}$ is the Kronecker symbol. The expected result is then a straightforward consequence of 
$$
\lim_{\varepsilon\to0^+}\frac{P^*_{t,t+\varepsilon}(i,j)-\delta_{ij}}{\varepsilon}=\alpha Q_t(i,j)\geq A_{ij}\psi(X(t)),
$$
and of Proposition \ref{prop:SDDI}. The proof in the case of Model \eqref{eq:MT} is identical except at the end where we lower bound $Q_t(i,j)$ as below
$$
Q_t(i,j)=\frac{A_{ij}\psi(\norm{x_i(t)-x_j(t)}_2)}{a_i+\sum_{k\ne i}A_{ik}\psi\left(\norm{x_i(t)-x_k(t)}_2\right)}\geq B_{ij}\psi(X(t)).
$$
\begin{flushright}
    \qed
\end{flushright}
\end{proof}

\subsection{The Hierarchical Leadership case}\label{subsec:prfHL}

The key idea of the proof of Theorem \ref{thm:HL} is that, if $(Y^{(T)}_t)_{t\in[0,T]}$ is a Markov jump process defined as in Theorem \ref{thm:probint} and if $A$ satisfies Assumption
\ref{ass:HL}, then the event $"Y^{(T)}_t=1"$ happens almost surely after $H$ jumps (i.e if $t$ is superior to the instant of the $H$-th jump). \\

The proof of Theorem \ref{thm:HL} needs an explicit construction of a Markov jump process on an interval $[0,T]$. Such a construction is given in \cite[Lemma 2.1 and Theorem 2.2]{feinberg2014solutions} for the general case of a Markov jump process with values in a Borel space. For sake of completeness, we give the following lemma which provides an equivalent construction adapted to our simpler case. \\

Let $(Q_t)_{t\geq 0}$ be a family of $N\times N$-transition rate matrix. Let $T>0$ be a positive real number and for all $(i,t)\in\intint{1}{N}\times[0,T]$, let us denote $q_i(t)=\alpha\sum_{j\ne i}Q_{T-t}(i,j)$. Let us also defined 
$$
\mathcal{R}^{(T)}_i(s,t)=\left\{\begin{aligned}
    &\left(\int_0^.q_i(s+u)du\right)^{-1}(t) &\text{if}\ t<\int_0^{T-s}q_i(s+u)du, \\
    &+\infty &\text{else}.
\end{aligned}\right.
$$
Finally, let $\Pi_t(i,j)$ be the stochastic matrix defined by 
$$
\Pi_t(i,j)=\left\{\begin{aligned}
    &\frac{\alpha Q_{T-t}(i,j)}{q_i(t)}(1-\delta_{ij}) &\text{if}\ q_i(t)\ne0, \\
    &\delta_{ij} &\text{else},
\end{aligned}\right.
$$
Then, defining $J_0=0$, $Z_0\in\intint{1}{N}$ and $(\tau_n)_{n\geq1}$ be a sequence of iid, exponentially distributed random variables of rate $1$. Let $(J_n,Z_n)_{n\in\mathbb{N}}$ be defined recursively as follows:
\begin{equation}\label{eq:defJZ}
    \left\{\begin{aligned}
        &J_{n+1}=J_n+\mathcal{R}^{(T)}_{Z_n}(J_n,\tau_{n+1})\ \text{if}\ J_n<+\infty\quad,\quad J_{n+1}=+\infty\ \text{else}, \\
        &Z_{n+1}\sim\Pi_{J_{n+1}}(Z_n,.)\ \text{if}\ J_{n+1}<+\infty\quad,\quad Z_{n+1}=Z_n\ \text{else}. 
    \end{aligned}\right.
\end{equation}
\begin{lem}\label{lem:jumpprocess}
    The stochastic process $(Y_t)_{t\leq T}$ defined by
    $$
    Y_t=\sum_{n=0}^{+\infty}\mathbb{1}_{J_n\leq t<J_{n+1}}Z_n,
    $$
    is a Markov jump process on $[0,T]$ of generator $(\alpha Q_{T-t})_{t\leq T}$. 
\end{lem}

\begin{proof}[Proof of Theorem  \ref{thm:HL}]
We will show that \eqref{eq:propHL} holds with
$$
C(t,r)=\mathbb{P}\left(\Gamma_H>\alpha A_*\psi(r)t\right).
$$
Where $\Gamma_H$ follows the gamma distribution of parameter $(H,1)$, which is the distribution of a sum of $H$ independent exponential random variables of rate $1$. This function satisfies the first two properties of \eqref{eq:propHL} as $\psi$ and $t\mapsto\mathbb{P}(\Gamma_H>t)$ are decreasing and $\mathbb{P}(\Gamma_H>t)\underset{t\to+\infty}{\longrightarrow}0$. \\

Let $T>0$ be a positive real number and $(Y^{[i]})_{i\in\intint{1}{N}}$ be a sequence of jump processes of generator $(\alpha Q_{T-t})_{t\leq T}$ such that $Y_0^{[i]}=i$. We assume that all the $Y^{[i]}$ are constructed as in Lemma \ref{lem:jumpprocess} with the same exponential random variables $(\tau_n)_{n\geq1}$. Let also $J_n^{[i]}$ be defined as in \eqref{eq:defJZ}, be the $n$-th jump time of the process $Y^{[i]}$. Consequently, we have
$$
\mu(P^*_{0,T})\geq\inf_{i,j}\mathbb{P}(Y^{[i]}_T=1)\land\mathbb{P}(Y^{[j]}_T=1)\geq\inf_{i,j}\mathbb{P}(J^{[i]}_{h_i}\leq T)\land\mathbb{P}(J^{[j]}_{h_j}\leq T)=\inf_{i>1}\mathbb{P}(J^{[i]}_{h_i}\leq T),
$$
where we used consecutively the definition of $\mu(P^*_{0,T})$, $Y^{[i]}_T\sim P^*_{0,T}(i,.)$,  "$Y^{[i]}=1$ after $h_i$ jumps" happens almost surely and  $\mathbb{P}(Y^{[1]}_T=1)=1$. \\

Now, for all $i>1$ and $t\leq T$, as $\psi$ is decreasing, we have
$$
q_i(t)\geq \alpha A_*\psi\left(\sup_{u\leq T}X(u)\right),
$$
and so that for $s\leq T$ and $t\leq\int_0^{T-s}q_i(s+u)du$, we have
$$
\mathcal{R}_i^{(T)}(s,t)\leq \frac{t}{\alpha A_*\psi\left(\sup_{u\leq T}X(u)\right)}.
$$
Consequently, if we define 
$$
\Gamma^{(T)}_0=0\quad \text{and}\quad \Gamma^{(T)}_{n+1}=\Gamma^{(T)}_n+\frac{\tau_{n+1}}{\alpha A_*\psi\left(\sup_{u\leq T}X(u)\right)},
$$
then, we have
$$
J_n^{[i]}\leq \Gamma^{(T)}_n.
$$
and
$$
1-\mu(P^*_{0,T})\leq\mathbb{P}(\Gamma^{(T)}_H>T)=\mathbb{P}\left(\Gamma_H>\alpha A_*\psi\left(\sup_{u\leq T}X(u)\right)T\right)=C\left(T,\sup_{u\leq T}X(u)\right).
$$
Where $\Gamma_H\sim\Gamma(H,1)$. Finally, using the fact that $\Gamma_H$ is non-negative, we have
$$
\int_0^{+\infty}C(t,r)\ dt=\int_0^{+\infty}\mathbb{P}(\Gamma_H>\alpha A_*\psi(r)t)\ dt=\mathbb{E}\left(\frac{\Gamma_H}{\alpha A_*\psi(r)}\right)=\frac{H}{\alpha A_*\psi(r)},
$$
which allows us to conclude using Proposition \ref{prop:HL}. The proof for the case of Model \eqref{eq:MT} is identical except at the end where we use the following lower bound. Let us denote $A_i=\sum_{j\ne i}A_{ij}$ and $\rho=\psi(\sup_{u\leq T}X(u))$, we have 
$$
q_i(t) \geq \alpha \frac{A_i\rho}{a_i+A_i\rho}=\alpha \frac{A_i}{a_i+A_i}\frac{a_i\rho+A_i\rho}{a_i+A_i\rho}\geq \alpha B_*\frac{\Bar{a}\rho+A_*\rho}{\Bar{a}+A_*\rho}.
$$
\end{proof}

\begin{rem}
From the proof of Proposition \ref{prop:HL} and the expression of $\mathbb{P}(\Gamma_H>t)$, we get the following estimation of the rate of decay of $V(t)$:
$$
V(t)\leq V(0)\left(\sum_{n=0}^{H-1}\frac{(\omega t)^n}{n!}\right) e^{-\omega t}.
$$
where $\omega=\alpha A_*\psi(r_M)$ for Model \eqref{eq:CS} and $\omega=\alpha B_*\frac{\Bar{a}\psi(r_M)+A_*\psi(r_M)}{\Bar{a}+A_*\psi(r_M)}$ for Model \eqref{eq:MT}. Here, $r_M$ is the smallest real value which satisfies \eqref{eq:assHL}, which is also an upper bound of $\sup_{t\geq 0}X(t)$.
\end{rem}

\subsection{The General leadership case}\label{subsec:prfGL}

\begin{proof}[Proof of Theorem  \ref{thm:GL}]
Let us prove that \eqref{eq:propGL} holds with 
$$
C(t,r)=\left(1\land\frac{\alpha \hat{A}\psi(r)t}{D}\right)^De^{-\alpha \Bar{A}t}.
$$

Let $L_{ij}$ be the set of coalescence paths and $d_{ij}$ be the coalescence distance, defined as in Section \ref{subsec:GL}. Let $(i,j)\in\intint{1}{N}^2$ be two nodes such as $i\ne j$ and $(l^i,l^j)\in L_{ij}$ be two paths verifying $\max(\len(l^i),\len(l^j))=d_{ij}$. Let $a_{ij}$ be the last common element of $l^i$ and $l^j$. We can assume that $\len(l^i)=\len(l^j)=D$, even if it means duplicating some elements of $l^i$ or $l^j$. Thus, we have for all $s\leq t$,
$$
P^*_{s,t}(i,a_{ij})=\mathbb{P}\left(Y^{(t)}_{t-s}=a_{ij}\ |\ Y^{(t)}_0=i\right)\geq\prod^D_{k=1}\mathbb{P}\left(Y^{(t)}_{t_k}=l^i_k\ |\ Y^{(t)}_{t_{k-1}}=l^i_{k-1}\right),
$$
where $t_k=\frac{k}{D}(t-s)$. If $l^i_{k-1}=l^i_k$ then 
\begin{equation}\label{eq:eq1prfUCC}
    \mathbb{P}\left(Y^{(t)}_{t_k}=l^i_k\ |\ Y^{(t)}_{t_{k-1}}=l^i_{k-1}\right)\geq e^{-\int_{t_{k-1}}^{t_k}q_{l^i_{k-1}}(w)\ dw},
\end{equation}
where $q_i$ defined as in \eqref{lem:jumpprocess} for the matrix $(Q_{t-s})_{s\leq t}$. If $l^i_{k-1}\ne l^i_k$ then we have this time
\begin{equation}\label{eq:eq2prfUCC}
    \mathbb{P}\left(Y^{(t)}_{t_k}=l^i_k\ |\ Y^{(t)}_{t_{k-1}}=l^i_{k-1}\right)\geq\int_{t_{k-1}}^{t_k}e^{-\int_{t_{k-1}}^uq_{l^i_{k-1}}(w)\ dw}\alpha Q_{t-u}(l^i_{k-1},l^i_k)e^{-\int_u^{t_k}q_{l^i_k}(w)\ dw}du,
\end{equation}

See \cite[Equations (14), (15) and (16)]{feinberg2014solutions} for a proof of Inequalities \eqref{eq:eq1prfUCC} and \eqref{eq:eq2prfUCC}. \\ 

Now, we have for all $i\in\intint{1}{N}$
$$
q_i(w)\leq\alpha\Bar{A}.
$$
On the other hand, as $\psi$ is decreasing, we have for all $(i,j)\in\intint{1}{N}^2$ such that $A_{ij}>0$,
$$
Q_{t-u}(i,j)\geq \hat{A}\psi(\sup_{s\leq t}X(s)).
$$
As a result of Inequalities \eqref{eq:eq1prfUCC} and \eqref{eq:eq2prfUCC}, we have 
$$
\mathbb{P}(Y^{(t)}_{t-s}=a_{ij}\ |\ Y^{(t)}_0=i)\geq C\left(t-s,\sup_{u\leq t}X(u)\right),
$$

and thus, 
$$
\begin{aligned}
\mu(P^*_{s,t})&=\inf_{i,j}\sum_{k=1}^NP^*_{s,t}(i,k)\land P^*_{s,t}(j,k) \\
&\geq\inf_{i,j}P^*_{s,t}(i,a_{ij})\land P^*_{s,t}(j,a_{ij}) \\
&=\inf_{i,j}\mathbb{P}(Y^{(t)}_{t-s}=a_{ij}\ |\ Y^{(t)}_0=i)\land \mathbb{P}(Y^{(t)}_{t-s}=a_{ij}\ |\ Y^{(t)}_0=j) \\
&\geq C(t-s,\sup_{u\leq t}X(u)).
\end{aligned}
$$
Using Proposition \ref{prop:GL}, we conclude that finding $t_0>0$ and $r_0\geq X(0)$ such that 
$$
(r_0-X(0))>V(0)\frac{t_0}{\left(1\land\frac{\alpha \hat{A}\psi(r_0)t_0}{D}\right)^D}e^{\alpha\Bar{A}t_0}.
$$
As the right-hand side of the previous inequality is minimal for $t_0=\frac{D-1}{\alpha\Bar{A}}$, we conclude that \eqref{eq:GLCS} is a flocking condition for $D>1$. If $D=1$ then, as $t_0$ must be positive, we chose $t_0=\varepsilon>0$ and we take the limit as $\varepsilon$ tends to $0$. Therefore, finding $r_0\geq X(0)$ such that
$$
(r_0-X(0))>V(0)\frac{1}{\alpha \hat{A}\psi(x)},
$$
is a flocking condition, which leads to the same condition as before. \\

The proof in the case of Model \eqref{eq:MT} is identical but this time we have for all $i\in\intint{1}{N}$ and $\tau\leq t$, 

$$
q_i(w)\leq \alpha \Bar{B},
$$
and for all $u\leq t$ and $(i,j)\in E$, 
$$
Q_{t-u}(i,j)\geq\frac{A_{ij}\psi(\norm{x_j(t-u)-x_i(t-u)})}{a_i+\sum_{k\ne i}A_{ik}\psi(\norm{x_k(t-u)-x_i(t-u)})}\geq\frac{\hat{A}\psi(\sup_{s\leq t}X(s))}{K +\hat{A}\psi(\sup_{s\leq t}X(s))}.
$$
\begin{flushright}
    \qed
\end{flushright}
\end{proof}

\section*{Acknowledgments}

We would like to thank Bertrand Cloez for his insightful comments which greatly helped us in the construction of this article. \\

We would also like to thank the two anonymous reviewers for their comments which greatly contributed to the improvement of the article. \\ 

This work was supported by the French National Research Agency under the Investments for the Future Program, referred as ANR-16-CONV-0004.

\bibliographystyle{abbrv}
\bibliography{refs.bib}

\medskip

\textit{E-mail adress:} \texttt{adrien.cotil@inrae.fr}

\end{document}